\documentclass[11pt,reqno]{amsart}

\usepackage{amsmath, amssymb, amsthm}
\usepackage{mathptmx} 
\usepackage{comment}

\newtheorem{theorem}{Theorem}[section]
\newtheorem{lemma}[theorem]{Lemma}
\newtheorem{proposition}[theorem]{Proposition}
\newtheorem{corollary}[theorem]{Corollary}
\theoremstyle{definition}
\newtheorem{definition}[theorem]{Definition}
\theoremstyle{definition}
\newtheorem{notation}[theorem]{Notation}

\theoremstyle{remark}
\newtheorem{remark}[theorem]{Remark}
\numberwithin{equation}{section} 

\newcommand{\Diff}{\textrm{Diff}}
\newcommand{\diff}{\textrm{diff}}

\newcommand{\Sp}{\textrm{Sp}}
\newcommand{\fraksp}{\mathfrak{sp}}
\newcommand{\HS}{\textrm{HS}}

\newcommand{\sgn}{\textrm{sgn}}
\newcommand{\Aut}{\textrm{Aut}}
\newcommand{\End}{\textrm{End}}
\newcommand{\Arg}{\textrm{Arg}}
\newcommand{\Diag}{\textrm{Diag}}
\newcommand{\Tr}{\textrm{Tr}}
\newcommand{\Ric}{\textrm{Ric}}
\newcommand{\SpQHS}{\mathfrak{sp}_{{\mbox{\tiny{HS}}}}^{{\mbox{\tiny{Q}}}}}
\newcommand{\SpHS}{\mathfrak{sp}_{\mbox{\tiny{HS}}}}

\begin{document}

\title[$\Diff(S^{1})$, $\Sp(\infty)$, Brownian motion, Ricci curvation]
{Brownian motion and Ricci curvature on an infinite dimensional symplectic
group related to the diffeomorphism group of the circle}

\author{MANG WU}
\email{mangwu@math.ucr.edu}


\begin{abstract}
An embedding of the group $\Diff(S^{1})$ of orientation preserving
diffeomorphims of the unit circle $S^1$ into an infinite-dimensional
symplectic group, $\Sp(\infty)$, is studied. The authors prove that
this embedding is not surjective. A Brownian motion is constructed
on $\Sp(\infty)$. This study is motivated by recent work of
H.~Airault, S. Fang and P.~Malliavin.
The Ricci curvature of the infinite-dimensional symplectic group
is computed. The result shows that in almost all directions, the
Ricci curvature is negative infinity.
\end{abstract}

\maketitle

\section{Introduction}
The group $\Diff(S^{1})$ of orientation preserving diffeomorphims of
the unit circle $S^1$ has been extensively studied for a long time.
One of the goals of the research has been to construct and study the
properties of a Brownian motion on this group. In \cite{AirMall2001}
H.~Airault and P.~Malliavin considered an embedding of $\Diff(S^1)$
into an infinite-dimensional symplectic group.

This group, $\Sp(\infty)$, can be represented as a certain
infinite-dimensional matrix group. For such matrix groups, the
method of\cite{Gordina2000a, Gordina2005a} can be used to construct
a Brownian motion living in the group. This construction relies on
the fact that these groups can be embedded into a larger Hilbert
space of Hilbert-Schmidt operators. We use the same method to
construct a Brownian motion on $\Sp(\infty)$. One of the advantages
of Hilbert-Schmidt groups is that one can associate an
infinite-dimensional Lie algebra to such a group, and this Lie
algebra is a Hilbert space. This is not the case with $\Diff(S^1)$,
as an infinite-dimensional Lie algebra associated with $\Diff(S^1)$
is not a Hilbert space with respect to the inner product compatible
with the symplectic structure on $\Diff(S^1)$.

In the current paper, we describe in detail the embedding of
$\Diff(S^1)$ into $\Sp(\infty)$, and construct a Brownian motion on
$\Sp(\infty)$.
We also compute the Ricci curvature of $\Sp(\infty)$.
Our motivation comes from an attempt to use this
embedding to better understand Brownian motion in $\Diff(S^1)$ as
studied by H.~Airault, S. Fang and P.~Malliavin in a number of
papers (e.g. \cite{AirMall2001, AirMall2006, Fang2002,
FangLuo2007}). One of the main results of the paper is Theorem
\ref{t.4.6}, where we describe the embedding of $\Diff(S^1)$ into
$\Sp(\infty)$ and prove that the map is not surjective. Theorem
\ref{Main2} gives the construction of a Brownian motion on
$\Sp(\infty)$. In order for this Brownian motion to live in the
group we are forced to choose a non-$\operatorname{Ad}$-invariant
inner product on the Lie algebra of $\Sp(\infty)$. This fact has a
potential implication for this Brownian motion not to be
quasi-invariant for the appropriate choice of the Cameron-Martin
subgroup of $\Sp(\infty)$. This is in contrast to results in
\cite{AirMall2006}. The latter can be explained by the fact that the
Brownian motion we construct in Section \ref{s.6} lives in a
subgroup of $\Sp(\infty)$ whose Lie algebra is much smaller than the
full Lie algebra of $\Sp(\infty)$.

\section{The spaces $H$ and $\mathbb{H}_\omega$}

\begin{definition}\label{def.Omega}
Let $H$ be the space of complex-valued $C^\infty$ functions on the
unit circle $S^1$ with the mean value $0$. Define a bilinear form
$\omega$ on $H$ by
\[
\omega(u,v)=\frac{1}{2\pi}\int_0^{2\pi}uv'd\theta,
\hspace{.2in}\text{ for any } u,v\in H.
\]
\end{definition}

\begin{remark} By using integration by parts, we see that the form
$\omega$ is anti-symmetric, that is,  $\omega(u,v)=-\omega(v,u)$ for
any $u, v \in H$.
\end{remark}
Next we define an inner product $(\cdot,\cdot)_\omega$ on $H$ which
is compatible with the form $\omega$. First, we introduce a complex
structure on $H$, that is, a linear map $J$ on $H$ such that
$J^2=-id$. Then the inner product is defined by $(u,v)_\omega=\pm
\omega(u,J\bar{v})$, where the sign depends on the choice of $J$.
The complex structure $J$ in this context is called the Hilbert
transform.

\begin{definition}\label{def.L2Space}
Let $\mathbb{H}_0$ be the Hilbert space of complex-valued $L^2$
functions on $S^1$ with the mean value $0$ equipped with the inner
product
\[
(u,v) = \frac{1}{2\pi}\int_0^{2\pi}u\bar{v}d\theta, \hspace{.2in}
\text{ for any }  u,v\in\mathbb{H}_0.
\]
\end{definition}


\begin{notation}\label{Decomp}Denote $\hat{e}_n=e^{in\theta}, n \in \mathbb{Z}\backslash\{0\}$, and
$ \mathcal{B}_H =\left\{\hat{e}_n, \ n \in \mathbb{Z}\backslash\{0\}
\right\}.$ Let $\mathbb{H}^+$ and $\mathbb{H}^-$ be the closed
subspaces of $\mathbb{H}_0$ spanned by $\{\hat{e}_n:n>0\}$ and
$\{\hat{e}_n:n<0\}$, respectively. By $\pi^+$ and $\pi^-$ we denote
the projections of $\mathbb{H}_0$ onto subspaces $\mathbb{H}^+$ and
$\mathbb{H}^-$, respectively. For $u\in\mathbb{H}_0$, we can write
$u=u_++u_-$, where $u_+=\pi^+(u)$ and $u_-=\pi^-(u)$.
\end{notation}

\begin{definition}\label{def.J}
Define the \textbf{Hilbert transformation} $J$ on $\mathcal{B}_H$ by
\[
J:\hat{e}_n \mapsto i\sgn(n)\hat{e}_n
\]
where $\sgn(n)$ is the sign of $n$, and then extended by linearity
to $\mathbb{H}_0$.
\end{definition}

\begin{remark}
In the above definition, $J$ is defined on the space
$\mathbb{H}_{0}$. We need to address the issue whether it is
well--defined on the \emph{subspace} $H$. That is, if $J(H)\subseteq
H$. We will see that if we modify the space $H$ a little bit, for
example, if we let $C_0^1(S^1)$ be the space of complex-valued $C^1$
functions on the circle with mean value zero, then $J$ is \emph{not}
well--defined on $C_0^1(S^1)$. This problem really lies in the heart
of Fourier analysis. To see this, we need to characterize $J$ by
using the Fourier transform.
\end{remark}

\begin{notation}
For $u\in\mathbb{H}_0$, let $\mathcal{F}:u\mapsto\hat{u}$ be the
\textbf{Fourier transformation} with $\hat{u}(n)=(u,\hat{e}_n)$. Let
$\hat{J}$ be a transformation on $l^2(\mathbb{Z}\backslash\{0\})$
defined by $\big(\hat{J}\hat{u}\big)(n)=i\sgn(n)\hat{u}(n)$ for any
$\hat{u}\in l^2(\mathbb{Z}\backslash\{0\})$.
\end{notation}

The Fourier transformation $\mathcal{F}:\mathbb{H}_0 \to
l^2(\mathbb{Z}\backslash\{0\})$ is an isomorphism of Hilbert spaces,
and  $J=\mathcal{F}^{-1} \circ \hat{J} \circ \mathcal{F}$.

\begin{proposition}\label{JWell}
The Hilbert transformation $J$ is well--defined on $H$, that is
$J(H)\subseteq H$.
\end{proposition}

\begin{proof}
The key of the proof is the fact that functions in  $H$ can be
completely characterized by their Fourier coefficients. To be
precise, let $u\in\mathbb{H}_0$ be continuous. Then $u$ is
$C^\infty$ if and only if $\lim_{n\to\infty}n^k\hat{u}(n)=0$ for any
$k\in\mathbb{N}$. From this fact, it follows immediately that $J$ is
well--defined on $H$, because $J$ only changes the signs of the
Fourier coefficients of a function $u\in H$.

For completeness of exposition, we give a proof of this fact. Though
the statement is probably a standard fact in the Fourier analysis,
we found it proven only in one direction in \cite{KatznelsonBook}.

We first assume that $u$ is $C^\infty$. Then
$u(\theta)=u(0)+\int_0^\theta u'(t)dt$. So
\begin{align*}
\hat{u}(n)
&=\frac{1}{2\pi}\Big(\int_0^{2\pi}\int_0^{2\pi}u'(t)\chi_{[0,\theta]}
dt\Big) e^{-in\theta}d\theta
=\frac{1}{2\pi}\int_0^{2\pi}\Big(\int_t^{2\pi}e^{-in\theta}d\theta\Big)u'(t)dt\\
&=-\frac{1}{2\pi in}\int_0^{2\pi}u'(t)-u'(t)e^{-int}dt
=\frac{\widehat{u'}(n)}{in},
\end{align*}
where we have used Fubini's theorem and the continuity of $u'$. Now,
$u'$ is itself $C^\infty$, so we can apply the procedure again. By
induction, we get
$\hat{u}(n)=\frac{\widehat{u^{(k)}}(n)}{\left(in\right)^{k}}$. But
from the general theory of Fourier analysis,
$\widehat{u^{(k)}}(n)\to 0$ as $n\to\infty$. Therefore
$n^k\hat{u}(n)\to 0$ as $n\to\infty$.

Conversely, assume $u$ is such that for any $k$,
$n^k\hat{u}(n)\to 0$ as $n\to\infty$.
Then the Fourier series of $u$ converges uniformly.
Also by assumption that $u$ is continuous, the Fourier series converges
to $u$ for all $\theta\in S^1$ (see Corollary I.3.1 in \cite{KatznelsonBook}).
So we can write $u(\theta) = \sum_{n\neq 0} \hat{u}(n) e^{in\theta}$.

Fix a point $\theta\in S^1$,
\[
u'(\theta)
= \left.\frac{d}{dt}\right|_{t=\theta}
\sum_{n\neq 0} \hat{u}(n) e^{int}
=\lim_{t\to\theta}\lim_{N\to\infty}
\sum_{n=-N}^{N} \hat{u}(n)
\frac{e^{int}-e^{in\theta}}{t-\theta}.
\]
Note that the derivatives of $\cos nt$ and $\sin nt$ are all bounded
by $|n|$. So by the mean value theorem, $|\cos nt-\cos n\theta|\le
|n||t-\theta|$, and $|\sin nt-\sin n\theta|\le |n||t-\theta|$. So
\[
\Big|\frac{e^{int}-e^{in\theta}}{t-\theta}\Big| \le 2|n|,
\hspace{.2in}\text{ for any } t,\theta\in S^1.
\]
Therefore, by the growth condition on the Fourier coefficients
$\hat{u}$, we have
\[
\lim_{N\to\infty}\sum_{n=-N}^{N}
\hat{u}(n) \frac{e^{int}-e^{in\theta}}{t-\theta}
\]
converges at the fixed $\theta\in S^1$ and the convergence is uniform in $t\in S^1$.
Therefore we can interchange the two limits, and obtain
\[
\Big(\sum_{n\neq 0} \hat{u}(n) e^{in\theta}\Big)'
=\sum_{n\neq 0} \hat{u}(n) in e^{in\theta},
\]
which means we can differentiate term by term. So the Fourier
coefficients of $u'$ are given by $\hat{u'}(n)=in\hat{u}(n)$.
Clearly, $\hat{u'}$ satisfies the same condition as $\hat{u}$:
$n^k\hat{u'}(n)\to 0$ as $n\to\infty$. By induction, $u$ is
$j$-times differentiable for any $j$. Therefore, $u$ is $C^\infty$.
\end{proof}

\begin{proposition}
Let $C_0^1(S^1)$ be the space of complex-valued $C^1$ functions on
the circle with the mean value zero. Then the Hilbert transformation
$J$ is \emph{not} well--defined on $C_0^1(S^1)$, i.e.,
$J(C_0^1(S^1))\nsubseteq C_0^1(S^1)$.
\end{proposition}

\begin{proof}
Let $C(S^1)$ be the space of continuous functions on the circle. In
 \cite{KatznelsonBook}, it is shown that there
exists a function in $C(S^1)$ such that the corresponding Fourier
series does not converges \emph{uniformly} \cite[Theorem
II.1.3]{KatznelsonBook}, and therefore there exists an $f\in C(S^1)$
such that $Jf\notin C(S^1)$ \cite[Theorem II.1.4]{KatznelsonBook}.
Now take $u=f-f_0$ where $f_0$ is the mean value of $f$. Then $u$ is
a continuous function on the circle with the mean value zero, and
$Ju$ is \emph{not} continuous.

Using Notation \ref{Decomp} let us write $u=u_+ + u_-$. Then we can
use the relation
\[
iu+Ju=2iu_+ \hspace{.1in}\mbox{and}\hspace{.1in} iu-Ju=2iu_-.
\]
to see that $u_+$ and $u_-$ are \emph{not} continuous. Integrating
$u=u_+ + u_-$, we have
\[
\int_0^t u(\theta)d\theta
=\int_0^t u_+(\theta)d\theta
+\int_0^t u_-(\theta)d\theta.
\]
Denote the three functions in the above equation by $v,v_1,v_2$.
By theorem I.1.6 in \cite{KatznelsonBook},
\[
\hat{v}(n)=\frac{\hat{u}(n)}{in},
\hspace{.1in}\mbox{and}\hspace{.1in}
\hat{v_1}(n)=\frac{\hat{u}_+(n)}{in},
\hat{v_2}(n)=\frac{1}{in}\hat{u}_-(n) \mbox{ for } n\neq 0.
\]

Let $g=v-v_0$ where $v_0$ is the mean value of $v$. Then $g\in C_0^1(S^1)$.
Write $g=g_+ + g_-$ \ref{Decomp}.
Then $g_+=v_1-(v_1)_0$ and $g_-=v_2-(v_2)_0$ where
$(v_1)_0$ and $(v_2)_0$ are the mean values of $v_1$ and $v_2$ respectively.
Then $g_+,g_-\notin C_0^1(S^1)$ since $v_1'=u_+,v_2'=u_-$ are \emph{not} continuous.

By the relation
\[
ig+Jg=2ig_+ \hspace{.1in}\mbox{and}\hspace{.1in} ig-Jg=2ig_-,
\]
we see that $Jg\notin C_0^1(S^1)$.
\end{proof}

\begin{notation}\label{nota.Inner}
Define an $\mathbb{R}$-bilinear form $(\cdot,\cdot)_\omega$ on $H$
by
\[
(u,v)_\omega=-\omega(u,J\bar{v}) \hspace{.2in}\text{ for any }
u,v\in H.
\]

\end{notation}

\begin{proposition}
$(\cdot,\cdot)_\omega$ is an inner product on $H$.
\end{proposition}

\begin{proof}
We need to check that $(\cdot,\cdot)_\omega$ satisfies the following
properties (1) $(\lambda u,v)_\omega=\lambda(u,v)_\omega$ for
$\lambda\in\mathbb{C}$; (2) $(v,u)_\omega=\overline{(u,v)_\omega}$;
(3) $(u,u)_\omega>0$ unless $u=0$.

(1) for $\lambda\in\mathbb{C}$,
\[
(\lambda u,v)_\omega =-\omega(\lambda u,J\bar{v})
=-\lambda\cdot\omega(u,J\bar{v}) =\lambda\cdot(u,v)_\omega.
\]

To prove (2) and (3), we need some simple facts:
$H^+=\pi^+(H)\subseteq H$ and $H^-=\pi^-(H)\subseteq H$, and
$H=H^+\oplus H^-$. If $u\in H^+,v\in H^-$, then $(u,v)=0$. If $u\in
H^+$, then $\bar{u}\in H^-, Ju=iu, Ju\in H^+$. If $u\in H^-$, then
$\bar{u}\in H^+, Ju=-iu, Ju\in H^-$. $J\bar{u}=\overline{Ju}$.
$\widehat{u'}(n)=in\hat{u}(n)$. In particular, if $u\in H^+$, then
$u'\in H^+$; if $u\in H^-$, then $u'\in H^-$.

(2)
By definition,
\begin{align*}
& (v,u)_\omega=-\omega(v,J\bar{u})=\omega(J\bar{u},v)
=\frac{1}{2\pi}\int (J\bar{u})v' d\theta \notag
\\ &
\overline{(u,v)_\omega}=-\overline{\omega(u,J\bar{v})}
=\overline{\omega(J\bar{v},u)} =\frac{1}{2\pi}\int
\overline{J\bar{v}}\bar{u}' d\theta =\frac{1}{2\pi}\int (Jv)\bar{u}'
d\theta. \label{e.2.1}
\end{align*}
Write $u=u_+ + u_-$ and $v=v_+ + v_-$ as in Notation \ref{Decomp}.
Using the above fact, we can show that the above two quantities are
equal to each other.

(3) Write $u=u_+ + u_-$, then
\[
(u,u)_\omega
=\frac{1}{2\pi}\int (-i\overline{u_+}u_+' + i\overline{u_-}u_-') d\theta
=\sum_{n\neq 0}|n||\hat{u}(n)|^2.
\]
Therefore, $(u,u)_\omega>0$ unless $u=0$.
\end{proof}

\begin{definition}\label{def.OmegaBasis}
Let $\mathbb{H}_\omega$ be the completion of $H$ under the norm
$\|\cdot\|_\omega$ induced by the inner product
$(\cdot,\cdot)_\omega$. Define
\[
\mathcal{B}_\omega=
\left\{\tilde{e}_n=\frac{1}{\sqrt{n}}e^{in\theta},
n>0 \right\} \cup \left\{\tilde{e}_n=\frac{1}{i\sqrt{|n|}}e^{in\theta},
n<0 \right\}.
\]
\end{definition}

\begin{remark}\label{Completion}
$\mathbb{H}_\omega$ is a Hilbert space. Also the norm
$\|\cdot\|_\omega$ induced by the inner product
$(\cdot,\cdot)_\omega$ is \emph{strictly} stronger than the norm
$\|\cdot\|$ induced by the inner product $(\cdot,\cdot)$. So
$\mathbb{H}_\omega$ can be identified as a \emph{proper} subspace of
$\mathbb{H}_0$. The inner product $(\cdot,\cdot)_\omega$ or the norm
induced by it is sometimes called the $H^{1/2}$ metric or the
$H^{1/2}$ norm on the space $H$.

One can verify that $\mathcal{B}_\omega$ is an orthonormal basis of
$\mathbb{H}_\omega$. From the definition of the inner product
$(\cdot,\cdot)_\omega$, we have the relation
$\omega(u,v)=(u,\overline{Jv})_\omega$ for any $u,v\in H$. This can
be used to \emph{extend} the form $\omega$ to $\mathbb{H}_\omega$.

Finally, from the non--degeneracy of the inner
product $(\cdot,\cdot)_\omega$, we see that the form
$\omega(\cdot,\cdot)$ on $\mathbb{H}_\omega$ is also
non--degenerate.
\end{remark}

\section{An infinite-dimensional symplectic group}

\begin{definition}\label{def.BoundedOp}
Let $B(\mathbb{H}_\omega)$ be the space of \textbf{bounded
operators} on $\mathbb{H}_\omega$ equipped with the operator norm.
For an operator $A\in B(\mathbb{H}_\omega)$

\begin{enumerate}

\item
suppose $\bar{A}$ is an operator on $\mathbb{H}_\omega$ satisfying
$\bar{A}u=\overline{A\bar{u}}$ for any $u\in \mathbb{H}_\omega$,
then $\bar{A}$ is the \textbf{conjugate} of $A$;

\item suppose $A^\dag$ is an operator on $\mathbb{H}_\omega$ satisfying
$(Au,v)_\omega=(u,A^\dag v)_\omega$ for any $u,v\in\mathbb{H}_\omega$, then $A^\dag$ is the
\textbf{adjoint} of $A$;

\item  then $A^T=\bar{A}^\dag$ is the \textbf{transpose} of $A$;

\item  suppose $A^\#$ is an operator on $\mathbb{H}_\omega$ satisfying $\omega(Au,v)=\omega(u,A^\# v)$ for any
$u,v\in\mathbb{H}_\omega$, then $A^\#$ is the \textbf{symplectic
adjoint} of $A$.

\item $A$ is said to \textbf{preserve the form} $\omega$ if
$\omega(Au,Av)=\omega(u,v)$ for any $u,v\in\mathbb{H}_\omega$.
\end{enumerate}
\end{definition}

In the orthonormal basis $\mathcal{B}_\omega$, an operator $A\in
B(\mathbb{H}_\omega)$ can be represented by an infinite-dimensional
matrix, still denoted by $A$, with $(m,n)$th entry equal to
$A_{m,n}=(A\tilde{e}_n,\tilde{e}_m)_\omega$.

\begin{remark}\label{NegativeIdx}
If we represent an operator $A\in B(\mathbb{H}_\omega)$ by a matrix
$\{A_{m,n}\}_{m,n\in\mathbb{Z}\backslash\{0\}}$,
the indices $m$ and $n$ are allowed to be both positive and negative following
Definition \ref{def.OmegaBasis} of $\mathcal{B}_\omega$.
\end{remark}

The next proposition collects some simple facts about operations on
$B(\mathbb{H}_\omega)$ introduced in Definition \ref{def.BoundedOp}.

\begin{proposition}\label{SomeFacts1}
Let $A,B\in B(\mathbb{H}_\omega)$. Then
\begin{enumerate}
\item
$\overline{\tilde{e}_n}=i\tilde{e}_{-n}$,
$J\tilde{e}_n=i\sgn(n)\tilde{e}_n$,
$(\tilde{e}_n)'=in\tilde{e}_n$;

\item
$(\bar{A})_{m,n}=\overline{A_{-m,-n}}$;

\item
$(A^\dag)_{m,n}=\overline{A_{n,m}}$;

\item
$\bar{A}^\dag=\overline{A^\dag}$,
and
$(A^T)_{m,n}=A_{-n,-m}$;

\item
if $A=\bar{A}$, then $(A^\#)_{m,n}=\sgn(mn)\overline{A_{n,m}}$;

\item
$\overline{AB}=\bar{A}\bar{B}$,
$(AB)^\dag=B^\dag A^\dag$,
$(AB)^T=B^T A^T$,
$(AB)^\#=B^\# A^\#$;

\item
If $A$ is invertible,
then $\bar{A},A^T,A^\dag,A^\#$ are all invertible,
and $(\bar{A})^{-1}=\overline{A^{-1}}$,
$(A^T)^{-1}=(A^{-1})^T$,
$(A^\dag)^{-1}=(A^{-1})^\dag$,
$(A^\#)^{-1}=(A^{-1})^\#$;

\item
$(\pi^+)_{m,n}=\frac{1}{2}(\delta_{mn}+\sgn(m)\delta_{mn})$,
$(\pi^-)_{m,n}=\frac{1}{2}(\delta_{mn}-\sgn(m)\delta_{mn})$,
$\overline{\pi^+}=\pi^-$, $\overline{\pi^-}=\pi^+$,
$(\pi^+)^T=\pi^-$, $(\pi^-)^T=\pi^+$,
$(\pi^+)^\dag=\pi^+$, $(\pi^-)^\dag=\pi^-$;

\item
$J_{m,n}=i\sgn(m)\delta_{mn}$, $\bar{J}=J$, $J=i(\pi^+-\pi^-)$,
$J^T=-J$, $J^\dag=-J$, $J^2=-id$;

\item
$(A^\#)_{m,n}=\sgn(mn)A_{-n,-m}$.
\end{enumerate}
\end{proposition}

\begin{proof}
All of these properties can be checked by straight forward
calculations. We only prove (10).
\begin{align*}
&
(A^\#)_{m,n}
=(A^\#\tilde{e}_n,\tilde{e}_m)_\omega
=-\omega(A^\#\tilde{e}_n,J\overline{\tilde{e}_m})
=\omega(J\overline{\tilde{e}_m},A^\#\tilde{e}_n)
\\&
=\omega(AJ\overline{\tilde{e}_m},\tilde{e}_n)
=-\omega(\tilde{e}_n,AJ\overline{\tilde{e}_m})
=-\omega(\tilde{e}_n,J(-J)AJ\overline{\tilde{e}_m})
\\&
=-\omega(\tilde{e}_n,J\overline{(-J\bar{A}J\tilde{e}_m)}),
\end{align*}
where in the last equality we used property (6),
$\overline{AB}=\bar{A}\bar{B}$, and  property (9), $\bar{J}=J$, so
that
$\overline{-J\bar{A}J\tilde{e}_m}=-\bar{J}\bar{\bar{A}}\bar{J}\overline{\tilde{e}_m}
=-JAJ\overline{\tilde{e}_m}$. Therefore,
\begin{align*}
&
(A^\#)_{m,n}
=-\omega(\tilde{e}_n,J\overline{(-J\bar{A}J\tilde{e}_m)})
=(\tilde{e}_n,-J\bar{A}J\tilde{e}_m)_\omega
=-(\tilde{e}_n,J\bar{A}J\tilde{e}_m)_\omega
\\&
=-(J^\dag\tilde{e}_n,\bar{A}J\tilde{e}_m)_\omega
=-(-J\tilde{e}_n,\bar{A}J\tilde{e}_m)_\omega
=(i\sgn(n)\tilde{e}_n,\bar{A}i\sgn(m)\tilde{e}_m)_\omega
\\&
=\sgn(mn)(\tilde{e}_n,\bar{A}\tilde{e}_m)_\omega
=\sgn(mn)\overline{(\bar{A}\tilde{e}_m,\tilde{e}_n)_\omega}
=\sgn(mn)\overline{(\bar{A})_{n,m}}
\\&
=\sgn(mn)A_{-n,-m}.
\end{align*}
\end{proof}

\begin{notation}\label{nota.BlockMatrix}
For $A\in B(\mathbb{H}_\omega)$, let $a=\pi^+ A \pi^+$, $b=\pi^+ A
\pi^-$, $c=\pi^- A \pi^+$, and $d=\pi^- A \pi^-$, where
$a:\mathbb{H}_\omega^+\to\mathbb{H}_\omega^+$,
$b:\mathbb{H}_\omega^-\to\mathbb{H}_\omega^+$,
$c:\mathbb{H}_\omega^+\to\mathbb{H}_\omega^-$,
$d:\mathbb{H}_\omega^-\to\mathbb{H}_\omega^-$. Then $A=a+b+c+d$ can
be represented as the following block matrix
\[
\left(
\begin{array}{ll}
a &b\\
c &d
\end{array}
\right).
\]
\end{notation}

If $A, B\in B(\mathbb{H}_\omega)$, then the block matrix
representation for $AB$ is exactly the multiplication of block
matrices for $A$ and $B$.

\begin{proposition}
Suppose $A\in B(\mathbb{H}_\omega)$ with the matrix $\{A_{m,n}\}_{m,
n \in \mathbb{Z}\backslash\{0\}}$. Then the following are equivalent
\begin{enumerate}
\item
$A=\bar{A}$;

\item
if $u=\bar{u}$, then $Au=\overline{Au}$;

\item
$A_{m,n}=\overline{A_{-m,-n}}$ (\ref{NegativeIdx});

\item
as a block matrix, $A$ has the form
$
\left(
\begin{array}{ll}
a & b\\
\bar{b} &\bar{a}
\end{array}
\right).
$
\end{enumerate}
\end{proposition}

\begin{proof}
Equivalence of (1), (3) and (4) follows from
Proposition\ref{SomeFacts1} and Notation\ref{nota.BlockMatrix}.
First we  show that (1) is equivalent to (2).

[(1)$\Longrightarrow$(2)].
If $u=\bar{u}$, then $Au=\bar{A}u=\overline{A\bar{u}}=\overline{Au}$.

[(2)$\Longrightarrow$(1)]. Let
$u=\tilde{e}_n+\overline{\tilde{e}_n}$, and
$v=\tilde{e}_{-n}+\overline{\tilde{e}_{-n}}$. Then $u,v$ are
real-valued functions on the circle. Using Proposition
\ref{SomeFacts1} we have $\overline{\tilde{e}_n}=i\tilde{e}_{-n}$,
and therefore $Au=\overline{Au}$ and $Av=\overline{Av}$ imply
\begin{align*}
& A\tilde{e}_n + iA\tilde{e}_{-n} = \overline{A\tilde{e}_n} -
i\overline{A\tilde{e}_{-n}}
\\
& A\tilde{e}_n - iA\tilde{e}_{-n} = -\overline{A\tilde{e}_n} -
i\overline{A\tilde{e}_{-n}}.
\end{align*} Solving the above two
equations for $A\tilde{e}_n$, we have
\[
A\tilde{e}_n=-i\overline{A\tilde{e}_{-n}}
=\overline{A\overline{\tilde{e}_n}}=\bar{A}\tilde{e}_n
\]
with this being true for any $n\neq 0$, and so $A=\bar{A}$.
\end{proof}

\begin{proposition}\label{prop.Preserve}
Let $A\in B(\mathbb{H}_\omega)$.
The following are equivalent:

\begin{enumerate}
\item
$A$ preserves the form $\omega$;

\item
$\omega(Au,Av)=\omega(u,v)$ for any $u,v\in\mathbb{H}_\omega$;

\item
$\omega(A\tilde{e}_m,A\tilde{e}_n)=\omega(\tilde{e}_m,\tilde{e}_n)$
for any $m,n\neq 0$;

\item
$A^TJA=J$;

\item
$\sum_{k\neq 0}\sgn(mk)A_{k,m}A_{-k,-n}=\delta_{m,n}$ for any $m,n\neq0$.
\end{enumerate}
If we further assume that $A=\bar{A}$, then the following two are
equivalent to the above:

\begin{enumerate}
\item[(I)]
$a^T\bar{a}-b^\dag b=\pi^-$ and $a^T\bar{b}-b^\dag a=0$;

\item[(II)]
$\sum_{k\neq 0}\sgn(mk)A_{k,m}\overline{A_{k,n}}=\delta_{m,n}$ for any $m,n\neq0$.
\end{enumerate}

\end{proposition}

\begin{proof}
Equivalence of (1),(2) and (3) follows directly from Definition
\ref{def.BoundedOp}. Let us check the equivalency of (2) and (4).
First assume that (2) holds. By Remark \ref{Completion} we have
$\omega(u,v)=(u,J\bar{v})_\omega$ , and therefore
\[
\omega(Au,Av)=(Au,J\overline{Av})_\omega=(u,A^\dag J\overline{Av})_\omega.
\]
By assumption, $\omega(Au,Av)=\omega(u,v)$ for any
$u,v\in\mathbb{H}_\omega$. So by the non-degeneracy of the inner
product $(\cdot,\cdot)_\omega$, we have $A^\dag
J\overline{Av}=J\bar{v}$ for any $v\in\mathbb{H}_\omega$. By
definition of $\bar{A}$, we have $\overline{Av}=\bar{A}\bar{v}$. So
$A^\dag J\bar{A}\bar{v}=J\bar{v}$ for any $v\in\mathbb{H}_\omega$,
or $A^\dag J\bar{A}=J$. Taking conjugation of both sides and using
$\bar{J}=J$, we see that $A^T JA=J$.

Every step above is reversible, therefore we have implication in the
other direction as well.

Now we check the equivalency of (3) and (5). First, by Remark
\ref{Completion} $\omega(u,v)=(u,J\bar{v})_\omega$
 and Proposition \ref{SomeFacts1}
\[
\omega(\tilde{e}_m,\tilde{e}_n)
=(\tilde{e}_m,J\overline{\tilde{e}_n})_\omega
=-\sgn(m)\delta_{m,-n}.
\]
On the other hand, by the continuity of the form
$\omega(\cdot,\cdot)$ in both variables, we have
\begin{align*}
&
\omega(A\tilde{e}_m,A\tilde{e}_n)
=\omega\Big(\sum_k A_{k,m}\tilde{e}_k,\sum_k A_{l,n}\tilde{e}_l\Big)
\\ &
=\sum_{k,l}A_{k,m}A_{l,n}(-\sgn(k))\delta_{k,-l} =-\sum_k
\sgn(k)A_{k,m}A_{-k,n}.
\end{align*}
Now assuming
$\omega(A\tilde{e}_m,A\tilde{e}_n)=\omega(\tilde{e}_m,\tilde{e}_n)$,
we have
\[
-\sum_k \sgn(k)A_{k,m}A_{-k,n}=-\sgn(m)\delta_{m,-n},
\text{ for any } m,n\neq 0.
\]
By multiplying by $\sgn(m)$ both sides, and replacing $-n$ with $n$,
we get (5). Conversely, note that every step above is reversible,
therefore we have implication in the other direction.

We have proved equivalence of (1)-(5). Now assume $A=\bar{A}$. To
prove equivalence of (4) and (I), just notice that as block
matrices, $A,A^T$ and $J$ have the form
\[
\left(
\begin{array}{ll}
a & b\\
\bar{b} & \bar{a}
\end{array}
\right),
\hspace{.2in}
\left(
\begin{array}{ll}
a^\dag & b^T\\
b^\dag & a^T
\end{array}
\right),
\hspace{.1in}\mbox{and}\hspace{.1in}
i\left(
\begin{array}{ll}
\pi^+ & 0\\
0 & -\pi^-
\end{array}
\right).
\]
Equivalence of (5) and (II) follows 
from the relation
$A_{-k,-n}=\overline{A_{k,n}}$.
\end{proof}

\begin{proposition}\label{prop.Invertible}
Let $A\in B(\mathbb{H}_\omega)$.
If $A$ preserves the form $\omega$, then the following are equivalent:
\begin{enumerate}
\item
$A$ is invertible.

\item
$AJA^T=J$.

\item
$A^T$ preserves the form $\omega$.

\item
$\sum_k \sgn(mk)A_{m,k}A_{-n,-k}=\delta_{m,n}$ for any $m,n\neq0$.
\end{enumerate}

If we further assume that $A=\bar{A}$, then the following are
equivalent to the above:

\begin{enumerate}
\item[(I)]
$\bar{a}a^T-\bar{b}b^T=\pi^-$ and $\bar{b}a^\dag-\bar{a}b^\dag=0$.

\item[(II)]
$\sum_k \sgn(mk)A_{m,k}\overline{A_{n,k}}=\delta_{m,n}$ for any $m,n\neq0$.
\end{enumerate}

\end{proposition}

\begin{proof} We will use several times the fact that if $A$ preserves $\omega$, then
$A^TJA=J$.

[(1)$\Rightarrow$(2)] Multiplying on the left by $(A^T)^{-1}$ and
multiplying on the right by $A^{-1}$ both sides, we get
$J=(A^T)^{-1}JA^{-1}$, and so $(A^{-1})^TJA^{-1}=J$. Taking inverse
of both sides, and using $J^{-1}=-J$, we have $A^TJA=J$.

[(2)$\Rightarrow$(1)]  As $J$ is injective, so is $A^TJA$, and
therefore $A$ is injective. On the other hand, by assumption
$AJA^T=J$. As $J$ is surjective, so $AJA^T$ is surjective too. This
implies that $A$ is surjective, and therefore $A$ is invertible.

Equivalence of (2) and (3) follows from  $(A^T)^T=A$ and Proposition
\ref{prop.Preserve}. Equivalence of (3) and (4) follows directly
from Proposition \ref{prop.Preserve} and the fact that
$(A^T)_{m,n}=A_{-n,-m}$.

Now assume that $A=\bar{A}$. Then equivalence of (3) and (I)can be
checked by using multiplication of block matrices as in the proof of
Proposition \ref{prop.Preserve}. Finally (4) is equivalent to (II)
as if $A=\bar{A}$, then $A_{-m,-n}=\overline{A_{m,n}}$.
\end{proof}

\begin{corollary}\label{AASharp}
Let $A\in B(\mathbb{H}_\omega)$ and $A=\bar{A}$.
Then the following are equivalent:
\begin{enumerate}
\item
$A$ preserves the form $\omega$ and is invertible;
\item
$A^\#A=A^\#A=id$;
\end{enumerate}
\end{corollary}

\begin{proof}
By Proposition \ref{SomeFacts1}
\begin{align*}
&(A^\#A)_{m,n}=\sum_{k\neq0}(A^\#)_{m,k}A_{k,n}
=\sum_{k\neq0}\sgn(mk)A_{k,n}\overline{A_{k,m}},\\
&(AA^\#)_{m,n}=\sum_{k\neq0}A_{m,k}(A^\#)_{k,n}
=\sum_{k\neq0}\sgn(nk)A_{m,k}\overline{A_{n,k}}.
\end{align*}
Therefore, by (II) in Proposition \ref{prop.Preserve} and (II) in
Proposition \ref{prop.Invertible} we have equivalence.
\end{proof}

\begin{definition}\label{def.2Norm}
Define a
(semi)norm $\|\cdot\|_2$ on $B(\mathbb{H}_\omega)$ such that for
$A\in B(\mathbb{H}_\omega)$, $\|A\|_2^2=\Tr(b^\dag b)=\Vert b
\Vert_{HS}$, where $b=\pi^+ A \pi^-$. That is, the norm $\|A\|_2$ is
just the Hilbert-Schmidt norm of the block $b$.
\end{definition}

\begin{definition}\label{def.SPInf}
An \textbf{infinite-dimensional symplectic group} $\Sp(\infty)$ is
the set of bounded operators $A$ on $H$ such that
\begin{enumerate}
\item $A$ is invertible;
\item $A=\bar{A}$;
\item $A$ preserves the form $\omega$;
\item $\|A\|_2<\infty$.
\end{enumerate}
\end{definition}

\begin{remark}\label{def.SPInf1}
If $A$ is a bounded operator on $H$, then $A$ can be extended to a
bounded operator on $\mathbb{H}_\omega$. Therefore, we can
equivalently define $\Sp(\infty)$ to be the set of operators $A\in
B(\mathbb{H}_\omega)$ such that
\begin{enumerate}
\item $A$ is invertible;
\item $A=\bar{A}$;
\item $A$ preserves the form $\omega$;
\item $\|A\|_2<\infty$.
\item $A$ is invariant on $H$, i.e., $A(H)\subseteq H$.
\end{enumerate}
\end{remark}

\begin{remark}\label{def.SPInf2}
By Corollary \ref{AASharp}, the definition of $\Sp(\infty)$ is also
equivalent to
\begin{enumerate}
\item $A=\bar{A}$;
\item $A^\#A=AA^\#=id$;
\item $\|A\|_2<\infty$.
\end{enumerate}
\end{remark}

\begin{proposition}
$\Sp(\infty)$ is a group.
\end{proposition}

\begin{proof}
First we show that if $A\in \Sp(\infty)$, then $A^{-1} \in
\Sp(\infty)$. By the assumption on $A$, it is easy to verify that
$A^{-1}$ satisfies (1), (2), (3) and (5) in Remark \ref{def.SPInf1}.
We need to show that $A^{-1}$ satisfies the condition (4), i.e.
$\|A^{-1}\|_2<\infty$. Suppose
\[
A=\left(
\begin{array}{ll}
a &b\\
\bar{b} &\bar{a}
\end{array}
\right) \hspace{.1in}\mbox{and}\hspace{.1in} A^{-1}=\left(
\begin{array}{ll}
a' &b'\\
\overline{b'} &\overline{a'}
\end{array}
\right),
\]
where by our assumptions all blocks are bounded operators, and in
addition $b$ is a Hilbert-Schmidt operator. We want to prove $b'$ is
also a Hilbert-Schmidt operator. $AA^{-1}=I$ and $A^{-1}A=I$ imply
that
\[
ab^{\prime}=-b\overline{a'}, \hskip0.1in  a'a+b'\bar{b}=I.
\]
The last equation gives $a'ab'+b'\bar{b}b'=b'$, and so

\[
b'=a'ab'+b'\bar{b}b'=-a'b\overline{a'}+b'\bar{b}b'
\]
which is a Hilbert-Schmidt operator as $b$ and $\bar{b}$ are
Hilbert-Schmidt. Therefore $\|A^{-1}\|_2<\infty$ and $A^{-1}\in
\Sp(\infty)$.

Next we show that if $A,B\in \Sp(\infty)$, then $AB\in \Sp(\infty)$.
By the assumption on $A$ and $B$, it is easy to verify that $AB$
satisfies (1), (2), (3) and (5) in Remark \ref{def.SPInf1}. We need
to show that $AB$ satisfies the condition (4), i.e.
$\|AB\|_2<\infty$. Suppose
\[
A=\left(
\begin{array}{ll}
a &b\\
\bar{b} &\bar{a}
\end{array}
\right) \hspace{.1in}\mbox{and}\hspace{.1in} B=\left(
\begin{array}{ll}
c &d\\
\bar{d} &\bar{c}
\end{array}
\right),
\]
where all blocks are bounded, and $\Vert b\Vert_{HS}, \Vert
d\Vert_{HS}<\infty$. Then
\[
AB=\left(
\begin{array}{ll}
ac+b\bar{d} & ad+b\bar{c}\\
\bar{b}c+\bar{a}\bar{d} & \bar{b}d+\bar{a}\bar{c}
\end{array}
\right).
\]
Then
\[
\Vert AB \Vert_{2}^{2}=\|ad+b\bar{c}\|_{HS} \leqslant
\|ad\|_{2}+\|b\bar{c}\|_{HS}<\infty,
\] since both $ad$ and
$b\bar{c}$ are Hilbert-Schmidt operators. Therefore
$\|AB\|_2<\infty$ and $AB\in \Sp(\infty)$.
\end{proof}

\section{Symplectic Representation of $\Diff(S^1)$}

\begin{definition}\label{d.4.1}
Let $\Diff(S^1)$ be the group of orientation preserving $C^\infty$
diffeomorphisms of $S^1$. $\Diff(S^1)$ acts on $H$ as follows
\[
(\phi.u)(\theta)=u(\phi^{-1}(\theta))
-\frac{1}{2\pi}\int_0^{2\pi} u(\phi^{-1}(\theta)) d\theta.
\]
\end{definition}
Note that if $u\in H$ is real-valued, then $\phi.u$ is real-valued
as well.

\begin{proposition}
The action of $\Diff(S^1)$ on $H$ gives a group homomorphism
\[
\Phi:\Diff(S^1)\to \Aut H
\]
defined by $\Phi(\phi)(u)=\phi.u$, for $\phi\in\Diff(S^1)$ and $u\in H$,
where $\Aut H$ is the group of automorphisms on $H$.
\end{proposition}

\begin{proof}
Let $u\in H$, then $\phi.u$ is  a $C^{\infty}$ function with the
mean value $0$, and so $\phi.u\in H$. It is also clear that
$\phi.(u+v)=\phi.u+\phi.v$ and $\phi.(\lambda u)=\lambda\phi.u$. So
$\Phi$ is well--defined as a map from $\Diff(S^1)$ to $\End H$, the
space of endomorphisms on $H$. Now let us check that $\Phi$ is a
group homomorphism. Suppose $\phi,\psi\in\Diff(S^1)$ and $u\in H$,
then
\begin{align*}
&
\Phi(\phi\psi)(u)(\theta)
=u\big((\phi\psi)^{-1}(\theta)\big)
-\frac{1}{2\pi}\int_0^{2\pi} u\big((\phi\psi)^{-1}(\theta)\big)d\theta
\\&
=u\big((\psi^{-1}\phi^{-1})(\theta)\big)
-\frac{1}{2\pi}\int_0^{2\pi} u\big((\psi^{-1}\phi^{-1})(\theta)\big)d\theta.
\end{align*}
On the other hand,
\begin{align*}
& \Phi(\phi)\Phi(\psi)(u)(\theta)
=\Phi(\phi)\left[u(\psi^{-1}(\theta)) -\frac{1}{2\pi}\int_0^{2\pi}
u(\psi^{-1}(\theta))d\theta\right]
\\&
=\Phi(\phi)\left[u(\psi^{-1}(\theta))\right]
=u\big((\psi^{-1}\phi^{-1})(\theta)\big)
-\frac{1}{2\pi}\int_0^{2\pi}
u\big((\psi^{-1}\phi^{-1})(\theta)\big)d\theta.
\end{align*}
So $\Phi(\phi\psi)=\Phi(\phi)\Phi(\psi)$. In particular, the image
of $\Phi$ is in the $\Aut H$.
\end{proof}

\begin{lemma}\label{l.4.4}
Any $\phi\in\Diff(S^1)$  preserves the form $\omega$, that is,
$\omega(\phi.u,\phi.v)=\omega(u,v)$ for any $ u,v\in H$.
\end{lemma}

\begin{proof}
By Definition \ref{d.4.1} $\phi.u=u(\psi)-u_0,\phi.v=v(\psi)-v_0$,
where $\psi=\phi^{-1}$ and $u_0,v_0$ are the constants. Then
\begin{align*}
\omega(\phi.u,\phi.v)
&=\omega(u(\psi)-u_0,v(\psi)-v_0)
\\&
=\frac{1}{2\pi}\int_0^{2\pi}
\big(u(\psi(\theta))-u_0\big)
\big(v(\psi(\theta))-v_0\big)' d\theta
\\&
=\frac{1}{2\pi}\int_0^{2\pi}
u(\psi)v'(\psi)\psi'(\theta)d\theta
-\frac{1}{2\pi}\int_0^{2\pi}
u_0v(\psi(\theta))d\theta
\\&
=\frac{1}{2\pi}\int_0^{2\pi}u(\psi)v'(\psi)d\psi
\\&
=\omega(u,v).
\end{align*}
\end{proof}

We are going to prove that a diffeomorphism $\phi\in\Diff(S^1)$ acts
on $H$ as a bounded linear map, and that $\Phi(\phi)$ is in
$\Sp(\infty)$. The next lemma is a generalization of a proposition
in a paper of G.~Segal\cite{Segal1981}.

\begin{lemma}\label{lem.BigLemma}
Let $\psi\neq id \in\Diff(S^1)$ and $\phi=\psi^{-1}$.
Let
\[
I_{n,m}=(\psi.e^{im\theta},e^{in\theta})
=\frac{1}{2\pi}\int_0^{2\pi}e^{im\phi-in\theta}d\theta.
\]
Then
\begin{enumerate}
\item
$\displaystyle \sum_{n>0,m<0}|n||I_{n,m}|^2 <\infty$, and
$\displaystyle \sum_{m>0,n<0}|n||I_{n,m}|^2 <\infty$.

\item
For sufficiently large $|m|$ there is a constant $C$ independent of
$m$ such that
\begin{equation}\label{e.4.1}
\sum_{n\neq0}|n||I_{n,m}|^2 < C|m|.
\end{equation}
\end{enumerate}
\end{lemma}

\begin{proof}
Let
\[
m_{\phi'}=\min\{\phi'(\theta)|\theta\in S^1\};
\hspace{.1in}\mbox{and}\hspace{.1in}
M_{\phi'}=\max\{\phi'(\theta)|\theta\in S^1\}.
\]
Since $\phi$ is a diffeomorphism, we have $0<m_{\phi'}<M_{\phi'}<\infty$.

Take four points $a,b,c,d$ on the unit circle such that
$a$ corresponds to $m_{\phi'}$ in the sense $\tan(a)=m_{\phi'}$,
$b$ corresponds to $M_{\phi'}$ in the sense $\tan(b)=M_{\phi'}$,
$c$ is opposite to $a$, i.e., $c=a+\pi$,
$d$ is opposite to $b$, i.e., $d=b+\pi$.
The four points on the circle are arranged in the counter-clockwise order,
and $0<a<b<\frac{\pi}{2}$, $\pi<c<d<\frac{3}{2}\pi$.

Let $\tau\in S^1$ such that $\tau\neq\frac{\pi}{4},\frac{5}{4}\pi$.
Define a function $\phi_\tau$ on $S^1$ by
\[
\phi_\tau(\theta)=
\frac{\cos\tau\cdot\phi(\theta)-\sin\tau\cdot\theta}
{\cos\tau-\sin\tau}.
\]

We will show that if $\tau\in (b,c)$ or $\tau\in(d,a)$,
then $\phi_\tau$ is an orientation preserving diffeomorphism of $S^1$,
where $(b,c)$ is the open arc from the point $b$ to the point $c$,
and $(d,a)$ is the open arc from the point $d$ to the point $a$.

Clearly $\phi_\tau$ is a $C^\infty$ function on $S^1$.
Also, $\phi_\tau(0)=0$ and $\phi_\tau(2\pi)=2\pi$.
Taking derivative with respect to $\theta$, we have
\[
\phi'_\tau(\theta)=
\frac{\cos\tau\cdot\phi'(\theta)-\sin\tau}
{\cos\tau-\sin\tau}.
\]
By the choice of $\tau$, we can prove that $\phi'_\tau(\theta)>0$.
Therefore, $\phi_\tau$ is an orientation preserving diffeomorphism as claimed.

Let $m,n\in\mathbb{Z}\backslash\{0\}$.
Let $\tau_{mn}=\Arg(m+in)$, i.e., the argument of the complex number
$m+in$, considered to be in $[0,2\pi]$.
Then we have $m\phi-n\theta=(m-n)\phi_{\tau_{mn}}$.

If $\tau_{mn}\in(b,c)$, then $\phi_{\tau_{mn}}$ is a diffeomorphism.
Let $\psi_{\tau_{mn}}=\phi_{\tau_{mn}}^{-1}$. Then
\[
I_{n,m}
=\frac{1}{2\pi}\int_0^{2\pi}e^{i(m-n)\phi_{\tau_{mn}}}d\theta
=\frac{1}{2\pi}\int_0^{2\pi}e^{i(m-n)\theta}
\psi_{\tau_{mn}}'(\theta)d\theta,
\]
where the last equality is by change of variable.
On integration by parts $k$ times, we have
\[
I_{n,m}
=\left(\frac{1}{i(m-n)}\right)^k\frac{1}{2\pi}\int_0^{2\pi}
e^{i(m-n)\theta}\psi_{\tau_{mn}}^{(k+1)}(\theta)d\theta.
\]

Let $\alpha=[\alpha_0,\alpha_1]$ be a closed arc contained in the arc $(b,c)$.
Let $S_\alpha$ be the set of all pairs of nonzero integers $(m,n)$ such that
$\alpha_0<\tau_{mn}<\alpha_1$, where $\tau_{mn}=\Arg(m+in)$.
We are going to consider an upper bound of the sum
$\sum_{(m,n)\in S_\alpha}|n||I_{n,m}|^2$.

For the pair $(m,n)$, if $|m-n|=p$, the condition
$\alpha_0<\tau_{mn}<\alpha_1$ gives us both an upper bound and a lower
bound for $n$:
\[
\frac{m_{\phi'}}{m_{\phi'}-1}p
\le n \le
\frac{M_{\phi'}}{M_{\phi'}-1}p.
\]
So $|n|\le C_1p$ where $C_1$ is a constant which does not depend on
the pair $(m,n)$. Also, the number of pairs $(m,n)\in S_\alpha$ such
that $|m-n|=p$ is bounded by $C_2p$ for some constant $C_2$. Let
$C_3=\max \Big\{ |\psi_\tau^{(k+1)}(\theta)| : \theta\in S^1,
\tau\in[\alpha_0,\alpha_1] \Big\}$. Then
\[
|I_{n,m}|\le
C_3\Big|\frac{1}{i(m-n)}\Big|^k\frac{1}{2\pi}\int_0^{2\pi}
e^{i(m-n)\theta}d\theta =C_3 p^{-k}.
\]
Therefore,
\begin{align*}
\sum_{(m,n)\in S}|n||I_{n,m}|^2
&=\sum_p\sum_{(m,n)\in S_\alpha; |m-n|=p}|n||I_{n,m}|^2
\\&
\le\sum_p C_1p\cdot C_3^2p^{-2k}\cdot C_2p =C_\alpha  \sum_p
p^{-(2k-2)},
\end{align*}
where the constant $C_\alpha$ depends on the arc $\alpha$.

Similarly, for a closed arc $\beta=[\beta_0,\beta_1]$ contained in
the arc $(d,a)$, we have
\[
\sum_{(m,n)\in S_\beta}|n||I_{n,m}|^2 \le C_\beta \sum_p
p^{-(2k-2)},
\]
where the constant $C_\beta$ depends on the arc $\beta$.

Now let $\alpha=[\frac{\pi}{2},\pi]$, and $\beta=[\frac{3}{2}\pi,2\pi]$.
Then $\alpha$ is contained in $(b,c)$ and $\beta$ is contained in $(d,a)$.
We have
\[
\sum_{n>0,m<0}|n||I_{n,m}|^2
=C_\alpha \cdot \sum_p p^{-(2k-2)}
<\infty
\]
and
\[
\sum_{n<0,m>0}|n||I_{n,m}|^2 =C_\beta \cdot \sum_p p^{-(2k-2)}
<\infty,
\]
which proves (1) of the lemma.

To prove (2), we let $\alpha=[\alpha_0,\alpha_1]$ be a closed arc
contained in the arc $(b,c)$ such that $b<\alpha_0<\frac{\pi}{2}$
and $\pi<\alpha_1<c$, and $\beta=[\beta_0,\beta_1]$ be a closed arc
contained in the arc $(d,a)$ such that $d<\beta_0<\frac{3}{2}\pi$
and $0<\beta_1<a$. Then we have
\[
\sum_{(m,n)\in S_\alpha}|n||I_{n,m}|^2 +\sum_{(m,n)\in
S_\beta}|n||I_{n,m}|^2 \leqslant  C_{\alpha\beta}
\]
for some constant $C_{\alpha\beta}$.

Let $m>0$ be sufficiently large, and $N_m$ be the largest integer
less than or equal to $m\tan(\alpha_0)$,
\[
\sum_{0<n\leqslant N_m}|I_{n,m}|^2 \leqslant
\sum_{n\neq0}|I_{n,m}|^2.
\]

Note that $I_{n,m}$ is the $n$th Fourier coefficient of
$\psi.e^{im\theta}$. Therefore,
\[
\sum_{n\neq0}|I_{n,m}|^2
=\|\psi.e^{im\theta}\|_{L^2}
\]
which is bounded by a constant $K$.
Therefore,
\[
\sum_{0<n\leqslant N_m}|n||I_{n,m}|^2 \leqslant
Km\tan\left(\alpha_0\right).
\]
On the other hand,
\[
\sum_{n<0}|n||I_{n,m}|^2 + \sum_{n>N_m}|n||I_{n,m}|^2 \leqslant
\sum_{(m,n)\in S_\alpha}|n||I_{n,m}|^2 + \sum_{(m,n)\in
S_\beta}|n||I_{n,m}|^2 =C_{\alpha\beta}.
\]
Therefore,
\[
\sum_{n\neq 0}|n||I_{n,m}|^2 \leqslant C_{\alpha\beta} +
Km\tan(\alpha_0) \leqslant m C_+,
\]
where $C_+$ can be chosen to be, for example,
$K\tan(\alpha_0)+C_{\alpha\beta}$, which is independent of $m$.

Similarly, for $m<0$ with sufficiently large $|m|$
\[
\sum_{n\neq 0}|n||I_{n,m}|^2 \leqslant m C_-.
\]

Let $C=\max \{C_+,C_-\}$. Then we have, for sufficiently
large $|m|$,
\[
\sum_{n\neq 0}|n||I_{n,m}|^2 \leqslant |m|C,
\]
which proves (2) of the lemma.
\end{proof}

\begin{lemma}\label{l.4.6}
For any $\psi\in\Diff(S^1)$, $\Phi(\psi)\in B(H)$, the space of
bounded linear maps on $H$. Moreover,
\[
\Vert\Phi(\psi)\Vert \leqslant C, \
\Vert\Phi(\psi)\Vert_{2}\leqslant C, \] where $C$ is the constant in
Equation \ref{e.4.1}.
\end{lemma}

\begin{proof}
First observe that the operator norm of $\Phi(\psi)$ is
\[
\|\Phi(\psi)\|=\sup\{\|\psi.u\|_\omega \hspace{.2cm}|\hspace{.2cm}
u\in H,\|u\|_\omega=1\}.
\]
For any $u\in H$, let $\hat{u}$ be its Fourier coefficients, that is
$\hat{u}(n)=(u,\hat{e}_n)$, and let $\tilde{u}$ be defined by
$\tilde{u}=(u,\tilde{e}_n)_\omega$
(\ref{nota.Inner},\ref{def.OmegaBasis}). It can be verified that the
relation between $\hat{u}$ and $\tilde{u}$ is: if $n>0$, then
$\tilde{u}(n)=\sqrt{n}\hat{u}(n)$; if $n<0$, then
$\tilde{u}(n)=i\sqrt{|n|}\hat{u}(n)$. We have
\[
\|u\|_\omega^2
=(u,u)_\omega
=(\tilde{u},\tilde{u})_{l^2}
=\sum_{n\neq0}|\tilde{u}(n)|^2
=\sum_{n\neq0}|n||\hat{u}(n)|^2.
\]
Let $\phi=\psi^{-1}$. We have $u(\phi)=\sum_{m\neq0}\hat{u}(m)e^{im\phi}$.
Using the notation $I_{n,m}$ (\ref{lem.BigLemma}), we have
\begin{align*}
& \|\psi.u\|_\omega^2 =\sum_{n\neq0}|n||\widehat{\psi.u}(n)|^2
=\sum_{n\neq0}|n|\Big| \frac{1}{2\pi}\int_0^{2\pi}
u(\phi(\theta))e^{-in\theta}d\theta \Big|^2
\\
&
=\sum_{n\neq0}|n|\Big| \frac{1}{2\pi}\int_0^{2\pi}
\sum_{m\neq0}\hat{u}(m)e^{im\phi} e^{-in\theta}d\theta \Big|^2
\\&
=\sum_{n\neq0}|n|\Big|\sum_{m\neq0}\hat{u}(m)
\frac{1}{2\pi}\int_0^{2\pi}e^{im\phi-in\theta}d\theta \Big|^2
\\&
=\sum_{n\neq0}|n|\Big|\sum_{m\neq0}\hat{u}(m) I_{n,m}\Big|^2 \\
& \leqslant \sum_{m,n\neq0}|n||\hat{u}(m)|^2|I_{n,m}|^2
=\sum_{m\neq0}|\hat{u}(m)|^2 \sum_{n\neq0}|n||I_{n,m}|^2
\\&
=\sum_{|m|\leqslant M_0}|\hat{u}(m)|^2 \sum_{n\neq0}|n||I_{n,m}|^2
+\sum_{|m|>M_0}|\hat{u}(m)|^2 \sum_{n\neq0}|n||I_{n,m}|^2,
\end{align*}
where the constant $M_0$ in the last equality is a positive integer
large enough so that we can apply part (2) of Lemma
\ref{lem.BigLemma}. It is easy to see that the first term in the
last equality is finite. For the second term we use Lemma
\ref{lem.BigLemma}
\[
\sum_{|m|>M_0}|\hat{u}(m)|^2 \sum_{n\neq0}|n||I_{n,m}|^2 \leqslant
C\sum_{|m|>M_0}|\hat{u}(m)|^2 |m| \leqslant C.
\]
Thus  for any  $u\in H$ with $\|u\|_\omega=1$, $\|\psi.u\|_\omega$
is uniformly bounded. Therefore, $\Phi(\psi)$ is a bounded operator
on $H$.

Now we can use Lemma \ref{lem.BigLemma} again to estimate the norm
$\Vert \Phi(\psi) \Vert_{2}$

\begin{align*}
& \|\Phi(\psi)\|_{2}
=\sum_{n>0,m<0}|(\psi.\tilde{e}_m,\tilde{e}_n)_\omega|^2
=\sum_{n>0,m<0}|n||(\psi.\hat{e}_m,\hat{e}_n)|^2
\\&
=\sum_{n>0,m<0}|n||I_{n,m}|^2 <\infty.
\end{align*}
\end{proof}

\begin{theorem}\label{t.4.6}
$\Phi:\Diff(S^1)\to \Sp(\infty)$ is a group homomorphism. Moreover,
$\Phi$ is injective, but not surjective.
\end{theorem}

\begin{proof}
Combining Lemma \ref{l.4.4} and Lemma \ref{l.4.6} we see that for
any diffeomorphism $\psi\in\Diff(S^1)$ the map $\Phi(\psi)$ is an
invertible bounded operator on $H$, it preserves the form $\omega$,
and $\|\Phi(\psi)\|_{2}<\infty$. In addition, by our remark after
Definition \ref{d.4.1} $\psi.u$ is real-valued, if $u$ is
real-valued. Therefore, $\Phi$ maps $\Diff(S^1)$ into $\Sp(\infty)$.

Next, we first prove that $\Phi$ is injective. Let
$\psi_1,\psi_2\in\Diff(S^1)$, and denote $\phi_1=\psi_1^{-1},
\phi_2=\psi_2^{-1}$. Suppose $\Phi(\psi_1)=\Phi(\psi_2)$, i.e.
$\psi_1.u=\psi_2.u$, for any $u\in H$. In particular,
$\psi_1.e^{i\theta}=\psi_2.e^{i\theta}$. Therefore
\[
e^{i\phi_1}-C_1=e^{i\phi_2}-C_2,
\]
where $C_1=\frac{1}{2\pi}\int_0^{2\pi}e^{i\phi_1}d\theta$, and
$C_2=\frac{1}{2\pi}\int_0^{2\pi}e^{i\phi_2}d\theta$. Note that
$e^{i\phi_1}$ and $e^{i\phi_2}$ have the same image as maps from
$S^{1}$ to $\mathbb{C}$. This implies $C_1=C_2$, since otherwise
$e^{i\phi_1}=e^{i\phi_2}+(C_1-C_2)$ and $e^{i\phi_1}$ and
$e^{i\phi_2}$ would have had different images. Therefore, we have
$e^{i\phi_1}=e^{i\phi_2}$. But the function $e^{i\tau}:S^1\to S^1$
is an injective function, so $\phi_1=\phi_2$. Therefore
$\psi_1=\psi_2$, and so $\Phi$ is injective.

To prove that $\Phi$ is not surjective, we will construct an
operator $A\in \Sp(\infty)$ which can not be written as $\Phi(\psi)$
for any $\psi\in\Diff(S^1)$. Let the linear map $A$ be defined by
the corresponding matrix $\{ A_{m, n}\}_{m, n \in \mathbb{Z}}$ with
the entries
\begin{align*}
& A_{1,1}=A_{-1,-1}=\sqrt{2}\\
& A_{1,-1}=i, A_{-1,1}=-i\\
& A_{m,m}=1, \hspace{.1in}\mbox{for }m \neq \pm1
\end{align*}
with all other entries  being $0$.

First we show that $A\in \Sp(\infty)$. For any $u\in H$, we can
write $u=\sum_{n\neq0}\tilde{u}(n)\tilde{e}_n$. Then $A$ acting on
$u$ changes  only $\tilde{e}_1$ and $\tilde{e}_{-1}$ . Therefore,
$Au\in H$, and clearly $A$ is a well--defined  bounded linear map on
$H$ to $H$. Moreover, $\|A\|_{2}<\infty$.  It is clear that
$A_{m,n}=\overline{A_{-m,-n}}$, and therefore $A=\bar{A}$ by
Proposition \ref{SomeFacts1}. Moreover, $A$ preserves the form
$\omega$ by  part(II) of Proposition \ref{prop.Preserve}, as

\[\sum_{k\neq 0}\sgn(mk)A_{k,m}\overline{A_{k,n}}
=\delta_{m,n}.
\]
Finally, $A$ is invertible, since $\{A_{k, m}\}_{m, n \in
\mathbb{Z}}$ is, with the inverse $\{B_{k, m}\}_{m, n \in
\mathbb{Z}}$ given by
\begin{align*}
& B_{1,1}=B_{-1,-1}=\sqrt{2}\\
& B_{1,-1}=-i, B_{-1,1}=i\\
& B_{m,m}=1, \hspace{.1in}\mbox{for }m \neq \pm1
\end{align*}
with all other entries being $0$. Next we show that
$A\not=\Phi(\psi)$ for any $\psi\in\Diff(S^1)$. First observe that
if we look at any basis element $\tilde{e}_1=e^{i\theta}$ as a
function from $S^1$ to $\mathbb{C}$, then the image of this function
lies on the unit circle. Clearly, when acted by a diffeomorphism
$\phi\in\Diff(S^1)$, the image of the function $\phi.e^{i\theta}$ is
still a circle with radius $1$. But if we consider $A\tilde{e}_1$ as
a function from $S^1$ to $\mathbb{C}$, we will show that the image
of the function $A\tilde{e}_1:S^1\to\mathbb{C}$ is not  a circle.
Therefore, $A\not=\Phi(\psi)$ for any $\psi\in\Diff(S^1)$. Indeed,
by definition of $A$ we have
\[
A\tilde{e}_1=\sqrt{2}\tilde{e}_1 - i \tilde{e}_{-1}.
\]
Let us write it as a function on $S^1$
\[
A\tilde{e}_1(\theta)=\sqrt{2}e^{i\theta}-e^{-i\theta}
=(\sqrt{2}-1)\cos\theta+i(\sqrt{2}+1)\sin\theta,
\]
and then we see that the image lies on an ellipse, which is not the
unit circle
\[ \frac{x^2}{(\sqrt{2}-1)^2}+\frac{y^2}{(\sqrt{2}+1)^2}=1. \]
\end{proof}

\section{The Lie algebra associated with $\Diff(S^1)$}

Let
$\diff(S^1)$ be the space of smooth vector fields on $S^1$. Elements
in $\diff(S^1)$ can be identified with smooth functions on $S^1$.
The space $\diff(S^1)$ is a Lie algebra with the following Lie
bracket
\[
[X,Y]=XY'-X'Y, \hskip0.1in X,Y\in\diff(S^1),
\]
where $X'$ and $Y'$ are derivatives with respect to $\theta$.

Let $X\in\diff(S^1)$, and $\rho_t$ be the corresponding flow of
diffeomorphisms. We define an action of $\diff(S^1)$ on $H$ as
follows: for $X\in\diff(S^1)$ and $u\in H$, $X.u$ is a function on
$S^1$ defined by
\[
(X.u)(\theta)=\left.\frac{d}{dt}\right|_{t=0} \left[
(\rho_t.u)(\theta) \right],
\]
where $\rho_t$ acts on $u$ via the representation
$\Phi:\Diff(S^1)\to \Sp(\infty)$.

The next proposition shows that the action is well--defined, and
also gives an explicit formula of $X.u$.

\begin{proposition}\label{prop.diffAction}
Let $X\in\diff(S^1)$.
Then
\[
(X.u)(\theta)
=u'(\theta)(-X(\theta))
-\frac{1}{2\pi}\int_0^{2\pi}u'(\theta)(-X(\theta))d\theta,
\]
that is, $X.u$ is the function $-u'X$ with the $0$th Fourier
coefficient replaced by $0$.
\end{proposition}

\begin{proof}
Let $\rho_t$ be the flow that corresponds to $X$, and $\lambda_t$ be the flow
that corresponds to $-X$. Then $\lambda_t$ is the inverse of $\rho_t$
for all $t$.
\[
(X.u)(\theta) =\left.\frac{d}{dt}\right|_{t=0} \left[
(\rho_t.u)(\theta) \right]
=\left.\frac{d}{dt}\right|_{t=0}\bigg[u(\lambda_t(\theta))
-\frac{1}{2\pi}\int_0^{2\pi}u(\lambda_t(\theta))d\theta\bigg].
\]
Using the chain rule, we have
\[
\left.\frac{d}{dt}\right|_{t=0}u(\lambda_t(\theta))
=u'(\theta)(-\widetilde{X}(\theta)),
\]
and
\[
\left.\frac{d}{dt}\right|_{t=0}\frac{1}{2\pi}
\int_0^{2\pi}u(\lambda_t(\theta))d\theta
=\frac{1}{2\pi}\int_0^{2\pi}u'(\theta)(-X(\theta))d\theta.
\]
\end{proof}

\begin{notation}
We consider $\diff(S^1)$ as a subspace of the space of real-valued $L^2$ functions on $S^1$.
The space of real-valued $L^2$ functions on $S^1$ has an orthonormal basis
\[
\mathcal{B}=\{X_{l}=\cos(m\theta),  Y_{k}=\sin(k\theta), l=0, 1,
..., k=1, 2,...\}
\]
which is contained in $\diff(S^1)$.
\end{notation}
Let us consider how these basis elements act on $H$.

\begin{proposition}
For  any $l=0, 1, ..., k=1, 2,...$ the basis elements $X_l,Y_k$ act
on $H$ as linear maps. In the basis $\mathcal{B}_\omega$ of $H$,
they are represented by infinite dimensional matrices with $(m,n)$th
entries equal to

\begin{align*}
(X_l)_{m,n}&=(X_l.\tilde{e}_n,\tilde{e}_m)_\omega=
s(m,n)\frac{1}{2}\sqrt{|mn|}(\delta_{m-n,l}+\delta_{n-m,l})\\
(Y_k)_{m,n}&=(Y_k.\tilde{e}_n,\tilde{e}_m)_\omega=
s(m,n)(-i)\frac{1}{2}\sqrt{|mn|}(\delta_{m-n,k}-\delta_{n-m,k})
\end{align*} where $m,n\neq0$,
\[
s(m,n)=\left\{
\begin{array}{ll}
-i & m,n>0\\
1  & m>0,n<0\\
1  & m<0,n>0\\
i  & m,n<0.
\end{array}
\right.
\]

\end{proposition}

\begin{proof}
By Proposition \ref{prop.diffAction} and a simple verification
depending on the signs of $m,n$ we see that
\begin{align*}
X_l.e^{in\theta}&=-ine^{in\theta}\cos(l\theta)
=-\frac{1}{2}in\left[e^{i(n+l)\theta}+e^{i(n-l)\theta}\right]\\
Y_k.e^{in\theta}&=-ine^{in\theta}\sin(k\theta)
=-\frac{1}{2}n\left[e^{i(n+k)\theta}-e^{i(n-k)\theta}\right].
\end{align*}
Indeed, recall that a basis element
$\tilde{e}_n\in\mathcal{B}_\omega$ has the form
\[
\tilde{e}_n=\left\{
\begin{array}{ll}
\frac{1}{\sqrt{n}}e^{in\theta} & n>0\\
\frac{1}{i\sqrt{|n|}}e^{in\theta} & n<0.
\end{array}
\right.
\]
Suppose $m,n>0$
\[
X_l.\tilde{e}_n=\frac{1}{\sqrt{n}}X_l.e^{in\theta}
=-\frac{1}{2}i\sqrt{n}\left[e^{i(n+l)\theta}+e^{i(n-l)\theta}\right],
\]
and
\[
(e^{i(n+l)\theta},\tilde{e}_m)_\omega=\sqrt{m}\delta_{m-n,k};
\hspace{.2in}
(e^{i(n-l)\theta},\tilde{e}_m)_\omega=\sqrt{m}\delta_{n-m,l}.
\]
Therefore,
\[
(X_l)_{m,n}=(X_l.\tilde{e}_n,\tilde{e}_m)_\omega=
(-i)\frac{1}{2}\sqrt{|mn|}(\delta_{m-n,l}+\delta_{n-m,l}).\\
\]
All other cases can be verified similarly.
\end{proof}

\begin{remark}
Recall that $\mathbb{H}_\omega$ is the completion of $H$ under
the metric $(\cdot,\cdot)_\omega$.
The above calculation shows that the trigonometric basis
$X_l,Y_k$ of $\diff(S^1)$ act on $\mathbb{H}_\omega$
as \emph{unbounded} operators. They are densely defined on the subspace
$H\subseteq\mathbb{H}_\omega$.
\end{remark}

\section{Brownian motion on $\Sp(\infty)$}\label{s.6}

\begin{notation}
As in \cite{AirMall2001}, let $\mathfrak{sp}\left(\infty\right)$  be
the set of infinite-dimensional matrices $A$  which can be written
as block matrices of the form
\[
\left(
\begin{array}{ll}
a &b\\
\bar{b} &\bar{a}
\end{array}
\right)
\]
such that $a+a^\dag=0$, $b=b^T$, and $b$ is a Hilbert-Schmidt
operator.
\end{notation}

\begin{remark}
 The set $\mathfrak{sp}\left(\infty\right)$ has a structure
of Lie algebra with the operator commutator as a Lie bracket, and we
associate this Lie algebra with the group $\Sp(\infty)$.
\end{remark}




\begin{proposition}\label{GEntry}
Let $\{A_{m,n}\}_{m, n \in \mathbb{Z}\backslash\{0\}}$ be the matrix
corresponding to an operator $A$. Then any $A\in
\mathfrak{sp}\left(\infty\right)$ satisfies (1)
$A_{m,n}=\overline{A_{-m,-n}}$; (2) $A_{m,n}+\overline{A_{n,m}}=0$,
for $m,n>0$; (3) $A_{m,n}=A_{-n,-m}$, for $m>0,n<0$.

Moreover, $A\in \mathfrak{sp}\left(\infty\right)$ if and only if (1)
$A=\bar{A}$; (2) $\pi^+A\pi^-$ is Hilbert-Schmidt; (3) $A+A^\#=0$.

\end{proposition}

\begin{proof}
The first part follows directly from definition of
$\mathfrak{sp}\left(\infty\right)$.  Then we can use this fact and
the formula for the matrix entries of $A^\#$ in Proposition
\ref{SomeFacts1} to prove the second part.
\end{proof}


\begin{definition}
Let $HS$ be the space of Hilbert-Schmidt matrices viewed as a real
vector space, and $\SpHS=\mathfrak{sp}\left(\infty\right) \cap HS$.
\end{definition}

The space $HS$ as a real Hilbert space has an orthonormal basis
\[
\mathcal{B}_{HS}=\{e_{mn}^{Re}:m,n\neq0\}\cup\{e_{mn}^{Im}:m,n\neq0\},
\]
where $e_{mn}^{Re}$ is a matrix with $(m,n)$-th entry $1$ all other
entries $0$, and $e_{mn}^{Im}$ is a matrix with $(m,n)$th entry $i$
all other entries $0$.

The space $\SpHS$ is a closed subspace of $HS$, and therefore
 a real Hilbert space. According to the symmetry of the
matrices in $\SpHS$, we define a projection $\pi:HS\to \SpHS$, such
that
\begin{align*}\label{Basis}
\pi(e_{mn}^{Re})
&=\frac{1}{2}\big(e_{mn}^{Re}-e_{nm}^{Re}+e_{-m,-n}^{Re}-e_{-n,-m}^{Re}\big),
&\mbox{if } \sgn(mn)>0\\
\pi(e_{mn}^{Im})
&=\frac{1}{2}\big(e_{mn}^{Im}+e_{nm}^{Im}-e_{-m,-n}^{Im}-e_{-n,-m}^{Im}\big),
&\mbox{if } \sgn(mn)>0\\
\pi(e_{mn}^{Re})
&=\frac{1}{2}\big(e_{mn}^{Re}+e_{-n,-m}^{Re}+e_{-m,-n}^{Re}+e_{n,m}^{Re}\big),
&\mbox{if } \sgn(mn)<0\\
\pi(e_{mn}^{Im})
&=\frac{1}{2}\big(e_{mn}^{Im}+e_{-n,-m}^{Im}-e_{-m,-n}^{Im}-e_{nm}^{Im}\big),
&\mbox{if } \sgn(mn)<0
\end{align*}

\begin{notation}
We choose $\mathcal{B}_{\SpHS}=\pi(\mathcal{B}_{HS})$ to be the
orthonormal basis of $\SpHS$.
\end{notation}

Clearly, if $A\in \SpHS$, then $|A|_{\SpHS}=|A|_{HS}$.

\begin{definition}
Let $W_{t}$ be a Brownian motion on $\SpHS$ which has the mean zero
and covariance $Q$, where $Q$ is assumed to be a positive symmetric
trace class operator on $H$. We further assume that $Q$ is diagonal
in the basis $\mathcal{B}_{\SpHS}$.
\end{definition}

\begin{remark}
$Q$ can also be viewed as a positive function on the set
$\mathcal{B}_{\SpHS}$, and the Brownian motion $W_{t}$ can be
written as
\begin{equation}
W_{t}=\sum_{\xi\in\mathcal{B}_{\SpHS}}\sqrt{Q(\xi)}B_{t}^\xi\xi,
\end{equation}
where $\{B_{t}^\xi\}_{\xi\in\mathcal{B}_{\SpHS}}$ are standard
real-valued mutually independent Brownian motions.
\end{remark}

Our goal now is to construct a Brownian motion on the group
$\Sp(\infty)$ using the Brownian motion $W_{t}$ on $\SpHS$. This is
done by solving the Stratonovich stochastic differential equation
\begin{equation}
\delta X_{t} = X_{t} \delta W_{t}.
\end{equation}
This equation can be written as the following It\^o stochastic
differential equation
\begin{equation}\label{ItoSDE}
dX_{t} = X_{t}dW_{t} + \frac{1}{2}X_{t}Ddt,
\end{equation}
where $D=\Diag(D_m)$ is a diagonal matrix with entries
\begin{equation}\label{DMatrix}
D_m=-\frac{1}{4}\sgn(m)\sum_k\sgn(k)\left[Q_{mk}^{Re}+Q_{mk}^{Im}\right]
\end{equation}
with $Q_{mk}^{Re}=Q(\pi(e_{mk}^{Re}))$ and $Q_{mk}^{Im}=Q(\pi(e_{mk}^{Im}))$.

\begin{notation}
Denote by $\SpQHS=Q^{1/2}(\SpHS)$ which is a subspace of $\SpHS$.
Define an inner product on $\SpQHS$ by $\langle u,v \rangle_{\SpQHS}
=\langle Q^{-1/2}u,Q^{-1/2}v \rangle_{\SpHS}$. Then
$\mathcal{B}_{\SpQHS}=\{\hat{\xi}=Q^{1/2}\xi:\xi\in\mathcal{B}_{\SpHS}\}$
is an orthonormal basis of the Hilbert space $\SpQHS$.
\end{notation}

\begin{notation}\label{L20Space}
Let $L_2^0$ be the space of Hilbert-Schmidt operators from $\SpQHS$
to $\SpHS$ with the norm
\[
|\Phi|_{L_2^0}^2 =\sum_{\hat{\xi}\in\mathcal{B}_{{\SpQHS}}}
|\Phi\hat{\xi}|_{\SpHS}^2
=\sum_{\xi,\zeta\in\mathcal{B}_{\SpHS}}Q(\xi)|\langle \Phi\xi,\zeta
\rangle_{\SpHS}|^2 =\Tr [\Phi Q \Phi^\ast],
\]
where $Q(\xi)$ means $Q$ evaluated at $\xi$ as a positive function
on $\mathcal{B}_{\SpHS}$.
\end{notation}

\begin{lemma}\label{L20Norm}
If $\Psi\in L({\SpHS},{\SpHS})$, a bounded linear operator from
$\SpHS$ to $\SpHS$, then $\Psi$ restricted on $\SpQHS$ is a
Hilbert-Schmidt operator from $\SpQHS$ to $\SpHS$, and
$|\Psi|_{L_2^0}\leqslant \Tr(Q)\|\Psi\|^2$, where $\|\Psi\|$ is the
operator norm of $\Psi$.
\end{lemma}

\begin{proof}
\begin{align*}
|\Psi|_{L_2^0}^2 &=\sum_{\hat{\xi}\in\mathcal{B}_{\SpQHS}}
|\Psi\hat{\xi}|_{\SpHS}^2
\leqslant \|\Psi\|^2 \sum_{\hat{\xi}\in\mathcal{B}_{\SpQHS}}|\hat{\xi}|_{\SpHS}^2\\
&=\|\Psi\|^2 \sum_{\xi\in\mathcal{B}_{\SpHS}}\langle
Q^{1/2}\xi,Q^{1/2}\xi\rangle_{\SpHS} =\|\Psi\|^2
\sum_{\xi\in\mathcal{B}_{\SpHS}}\langle Q\xi,\xi\rangle_{\SpHS}
=\|\Psi\|^2 \Tr(Q)
\end{align*}
\end{proof}

\begin{notation}\label{BF}
Define $B:{\SpHS}\to L_2^0$ by $B(Y)A=(I+Y)A$ for $A\in \SpQHS$, and
$F:{\SpHS}\to {\SpHS}$ by $F(Y)=\frac{1}{2}(I+Y)D$.
\end{notation}
Note that $B$ is well--defined by Lemma \ref{L20Norm}. Also $D\in
{\SpHS}$, and so $F(Y)\in {\SpHS}$ and $F$ is well--defined as well.
\begin{theorem}\label{Main1}
The stochastic differential equation
\begin{align}
&dY_t = B(Y_t)dW_{t} + F(Y_t)dt \label{e.6.1}\\
&Y_0=0 \notag
\end{align}
has a unique solution, up to equivalence,  among the processes satisfying
\[
P\left(\int_0^T |Y_s|_{\SpHS}^2 ds < \infty \right) = 1.
\]
\end{theorem}

\begin{proof}
To prove this theorem we will use Theorem 7.4 from the book by
G.~DaPrato and J.~Zabczyk \cite{DaPratoBook1992} as it has been done
in \cite{Gordina2000a, Gordina2005a}. It is enough to check
\begin{enumerate}
\item[1.]
$B$ is a measurable mapping.
\item[2.]
$|B(Y_1)-B(Y_2)|_{L_2^0} \leqslant C_1 |Y_1-Y_2|_{\SpHS}$ for
$Y_1,Y_2\in {\SpHS}$;
\item[3.]
$|B(Y)|_{L_2^0}^2 \leqslant K_1(1+|Y|_{\SpHS}^2)$ for any $Y\in
{\SpHS}$;
\item[4.]
$F$ is a measurable mapping.
\item[5.]
$|F(Y_1)-F(Y_2)|_{\SpHS} \leqslant C_2 |Y_1-Y_2|_{\SpHS}$ for
$Y_1,Y_2\in {\SpHS}$;
\item[6.]
$|F(Y)|_{\SpHS}^2 \leqslant K_2(1+|Y|_{\SpHS}^2)$ for any $Y\in
{\SpHS}$.
\end{enumerate}
Proof of 1. By the proof of 2, $B$ is a continuous mapping,
therefore it is measurable.

\noindent
Proof of 2.
\begin{align*}
&|B(Y_1)-B(Y_2)|_{L_2^0}^2 =\sum_{\hat{\xi}\in\mathcal{B}_{\SpQHS}}
|(Y_1-Y_2)\hat{\xi}|_{\SpHS}^2
=\sum_{\xi\in\mathcal{B}_{\SpHS}} Q(\xi)|(Y_1-Y_2)\xi|_{\SpHS}^2\\
&\leqslant \sum_{\xi\in\mathcal{B}_{\SpHS}}
Q(\xi)\|\xi\|^2|Y_1-Y_2|_{\SpHS}^2 \leqslant
\max_{\xi\in\mathcal{B}_{\SpHS}}\|\xi\|^2
\left(\sum_{\xi\in\mathcal{B}_{\SpHS}}Q(\xi)\right)|Y_1-Y_2|_{\SpHS}^2
\\
& =\Tr Q\left(\max_{\xi\in\mathcal{B}_{\SpHS}}\|\xi\|^2\right)
|Y_1-Y_2|_{\SpHS}^2 = C_1^2 |Y_1-Y_2|_{\SpHS}^2,
\end{align*}
where $\|\xi\|$ is the operator norm of $\xi$, which is uniformly
bounded for all $\xi\in\mathcal{B}_{\SpHS}$.

\noindent
Proof of 3.
\begin{align*}
|B(Y_1)|_{L_2^0}^2
&=\sum_{\hat{\xi}\in\mathcal{B}_{\SpQHS}}|(I+Y)\hat{\xi}|_{\SpHS}^2
=\sum_{\xi\in\mathcal{B}_{\SpHS}} Q(\xi)|(I+Y)\xi|_{\SpHS}^2\\
&\leqslant |(I+Y)\xi|_{\SpHS}^2 \sum_{\xi\in\mathcal{B}_{\SpHS}}
Q(\xi)\|\xi\|^2 \le(1+|Y|_{\SpHS}^2)\cdot K_1.
\end{align*}

\noindent Proof of 4. By the proof of 5, $F$ is a continuous
mapping, therefore it is measurable.

\noindent
Proof of 5.
\begin{align*}
|F(Y_1)-F(Y_2)|_{\SpHS}=|\frac{1}{2}(Y_1-Y_2)D|_{\SpHS}
\le\|\frac{1}{2}D\||Y_1-Y_2|_{\SpHS}
\end{align*}

\noindent
Proof of 6.
\[
|F(Y)|_{\SpHS}^2=|\frac{1}{2}(I+Y)D|_{\SpHS}^2
\le\|\frac{1}{2}D\|^2|I+Y|_{\SpHS}^2 \leqslant K_2(1+|Y|_{\SpHS}^2).
\]
\end{proof}

\begin{notation}\label{BFSharp}
Let $B^\#:\SpHS \to L_2^0$  be the operator  $B^\#(Y)A=A^\#(I+Y)$,
and $F^\#: \SpHS \to \SpHS$ be the operator
$F^\#(Y)=\frac{1}{2}D^\#(Y+I)$.
\end{notation}

\begin{proposition}\label{YSharp}
If $Y_t$ is the solution to the stochastic differential equation
\begin{align*}
& dX_t = B(X_t)dW_{t} + F(X_t)dt\\
&X_0=0,
\end{align*}
where $B$ and $F$ are defined in Notation \ref{BF}, then $Y_t^\#$ is
the solution to the stochastic differential equation
\begin{align}
& dX_t = B^\#(X_t)dW_{t} + F^\#(X_t)dt \label{e.6.2}\\
&X_0=0, \notag
\end{align}
where $B^\#$ and $F^\#$ are defined in Notation \ref{BFSharp}.
\end{proposition}

\begin{proof}
This follows directly from the property $(AB)^\#=B^\# A^\#$ for any
$A$ and $B$, which can be verified by using part (5) of Proposition
\ref{SomeFacts1}.
\end{proof}

\begin{lemma}\label{Trace}
Let $U$ and $H$ be real Hilbert spaces. Let $\Phi:U\to H$ be a
bounded linear map. Let $G:H\to H$ be a bounded linear map. Then
\[
\Tr_H(G\Phi\Phi^\ast)=\Tr_U(\Phi^\ast G\Phi)
\]
\end{lemma}

\begin{proof}
\begin{align*}
\Tr_H(G\Phi\Phi^\ast)
&=\sum_{i,j\in H; k\in U} G_{ij}\Phi_{jk}(\Phi^\ast)_{ki}
=\sum_{i,j\in H; k\in U} G_{ij}\Phi_{jk}\Phi_{ik}\\
\Tr_U(\Phi^\ast G\Phi)
&=\sum_{i,j\in H; k\in U}(\Phi^\ast)_{ki}G_{ij}\Phi_{jk}
=\sum_{i,j\in H; k\in U} G_{ij}\Phi_{jk}\Phi_{ik}.
\end{align*}
Therefore $\Tr_H(G\Phi\Phi^\ast)=\Tr_U(\Phi^\ast G\Phi)$.
\end{proof}

\begin{lemma}\label{Sum_xi}
\[
\sum_{\xi\in\mathcal{B}_{\SpHS}}\big(Q^{1/2}\xi\big)\big(Q^{1/2}\xi\big)^\#=-D
\]
\end{lemma}

\begin{proof}
If $\xi\in\mathcal{B}_{\SpHS}$, then
$\xi\in\mathfrak{sp}\left(\infty\right)$, so $\xi^\#=-\xi$. We will
use the fact that
\[
(e_{ij}^{Re}e_{kl}^{Re})_{pq}=\delta_{ip}\delta_{jk}\delta_{lq}
\]
where $e_{ij}^{Re}$ is the matrix with the $(i,j)$th entry being $1$
and all other entries being zero. Using this fact, we see
\begin{enumerate}
\item
for
$\xi=\frac{1}{2}\big(e_{mn}^{Re}-e_{nm}^{Re}+e_{-m,-n}^{Re}
-e_{-n,-m}^{Re}\big)$ with $\sgn(mn)>0$,
\[
\big(Q^{1/2}\xi\big)\big(Q^{1/2}\xi\big)^\# =-\frac{1}{4}Q_{mn}^{Re}
\left[-e_{mm}^{Re}-e_{nn}^{Re}-e_{-m,-m}^{Re}-e_{-n,-n}^{Re}\right]
\]
\item
for
$\xi=\frac{1}{2}\big(e_{mn}^{Im}+e_{nm}^{Im}-e_{-m,-n}^{Im}
-e_{-n,-m}^{Im}\big)$ with $\sgn(mn)>0$,
\[
\big(Q^{1/2}\xi\big)\big(Q^{1/2}\xi\big)^\# =-\frac{1}{4}Q_{mn}^{Im}
\left[-e_{mm}^{Re}-e_{nn}^{Re}-e_{-m,-m}^{Re}-e_{-n,-n}^{Re}\right]
\]
\item
for
$\xi=\frac{1}{2}\big(e_{mn}^{Re}+e_{-n,-m}^{Re}+e_{-m,-n}^{Re}
+e_{n,m}^{Re}\big)$ with $\sgn(mn)<0$,
\[
\big(Q^{1/2}\xi\big)\big(Q^{1/2}\xi\big)^\# =-\frac{1}{4}Q_{mn}^{Re}
\left[e_{mm}^{Re}+e_{nn}^{Re}+e_{-m,-m}^{Re}+e_{-n,-n}^{Re}\right]
\]
\item
for
$\xi=\frac{1}{2}\big(e_{mn}^{Im}+e_{-n,-m}^{Im}-e_{-m,-n}^{Im}
-e_{nm}^{Im}\big)$ with $\sgn(mn)<0$,
\[
\big(Q^{1/2}\xi\big)\big(Q^{1/2}\xi\big)^\# =-\frac{1}{4}Q_{mn}^{Im}
\left[e_{mm}^{Re}+e_{nn}^{Re}+e_{-m,-m}^{Re}+e_{-n,-n}^{Re}\right].
\]
\end{enumerate}

Each of the above is a diagonal matrix. The lemma can be proved by
looking at the diagonal entries of the sum.
\end{proof}

\begin{theorem}\label{Main2}
Let $Y_t$ be the solution to Equation \ref{e.6.1}. Then $Y_t+I\in
\Sp(\infty)$ for any $t>0$ with probability $1$.
\end{theorem}

\begin{proof}
The proof is adapted from papers by M.~Gordina \cite{Gordina2000a,
Gordina2005a}. Let $Y_t$ be the solution to Equation \eqref{e.6.1}
and $Y_t^\#$ be the solution to Equation \eqref{e.6.2}. Consider the
process $\textbf{Y}_t=(Y_t,Y_t^\#)$ in the product space
${\SpHS}\times {\SpHS}$. It satisfies the following stochastic
differential equation
\[
d\textbf{Y}_t=(B(Y_t),B^\#(Y_t^\#))dW + (F(Y_t),F^\#(Y_t^\#))dt.
\]

Let $G$ be a function on the Hilbert space ${\SpHS}\times {\SpHS}$
defined by $G(Y_1,Y_2)=\Lambda((Y_1+I)(Y_2+I))$, where $\Lambda$ is
a nonzero linear real bounded functional from ${\SpHS}\times
{\SpHS}$ to $\mathbb{R}$. We will apply It\^o's formula to
$G(\textbf{Y}_t)=G(Y_t,Y_t^\#)$. Then $(Y_t+I)(Y_t^\#+I)=I$ if and
only if $\Lambda((Y_t+I)(Y_t^\#+I)-I)=0$ for any $\Lambda$.

In order to use It\^o's formula we must verify that $G$ and the
derivatives $G_t$, $G_{\textbf{Y}}$, $G_{\textbf{YY}}$ are uniformly
continuous on bounded subsets of $[0,T]\times {\SpHS}\times
{\SpHS}$, where $G_{\textbf{Y}}$ is defined as follows
\[
G_{\textbf{Y}}(\textbf{Y})(\textbf{S}) =\lim_{\epsilon\to0}
\frac{G(\textbf{Y}+\epsilon\textbf{S})-G(\textbf{Y})}{\epsilon}
\hspace{.2in}\mbox{for any }\textbf{Y},\textbf{S}\in {\SpHS}\times
{\SpHS}
\]
and $G_{\textbf{YY}}$ is defined as follows
\[
G_{\textbf{YY}}(\textbf{Y})(\textbf{S}\otimes\textbf{T})
=\lim_{\epsilon\to0}
\frac{G_{\textbf{Y}}(\textbf{Y}+\epsilon\textbf{T})(\textbf{S})
-G_{\textbf{Y}}(\textbf{Y})(\textbf{S})}{\epsilon}
\]
for any $\textbf{Y},\textbf{S},\textbf{T}\in {\SpHS}\times {\SpHS}$.
Let us calculate $G_t$, $G_{\textbf{Y}}$, $G_{\textbf{YY}}$.
Clearly, $G_t=0$. It is easy to verify that for any
$\textbf{S}=(S_1,S_2)\in {\SpHS}\times {\SpHS}$
\[
G_{\textbf{Y}}(\textbf{Y})(\textbf{S})
=\Lambda(S_1(Y_2+I)+(Y_1+I)S_2)
\]
and for any $\textbf{S}=(S_1,S_2)\in {\SpHS}\times {\SpHS}$ and
$\textbf{T}=(T_1,T_2)\in {\SpHS}\times {\SpHS}$
\[
G_{\textbf{YY}}(\textbf{Y})(\textbf{S}\otimes\textbf{T})
=\Lambda(S_1T_2+T_1S_2).
\]
So the condition is satisfied.

We will use the following notation
\begin{align*}
&G_{\textbf{Y}}(\textbf{Y})(\textbf{S})
=\langle \bar{G}_{\textbf{Y}}(\textbf{Y}),\textbf{S} \rangle_{\SpHS\times\SpHS}\\
\noalign{\vskip .05 true in}
&G_{\textbf{YY}}(\textbf{Y})(\textbf{S}\otimes\textbf{T})= \langle
\bar{G}_{\textbf{YY}}(\textbf{Y})\textbf{S},\textbf{T}
\rangle_{\SpHS\times\SpHS},
\end{align*}
where $\bar{G}_{\textbf{Y}}(\textbf{Y})$ is an element of
${\SpHS}\times {\SpHS}$ corresponding to the functional
$G_{\textbf{Y}}(\textbf{Y})$ in $({\SpHS}\times {\SpHS})^\ast$ and
$\bar{G}_{\textbf{YY}}(\textbf{Y})$ is an operator on ${\SpHS}\times
{\SpHS}$ corresponding to the functional
$G_{\textbf{YY}}(\textbf{Y})\in (({\SpHS}\times
{\SpHS})\otimes({\SpHS}\times {\SpHS}))^\ast$.

Now we can apply It\^o's formula to $G(\textbf{Y}_t)$
\begin{align*}
G(\textbf{Y}_t)-G(\textbf{Y}_0)
=&\int_0^t \langle \bar{G}_{\textbf{Y}}(\textbf{Y}_s),
\big(B(Y_s)dW_s,B^\#(Y_s^\#)dW_s\big)\rangle_{\SpHS\times\SpHS}\\
+&\int_0^t \langle \bar{G}_{\textbf{Y}}(\textbf{Y}_s),
\big(F(Y_s),F^\#(Y_s^\#)\big)\rangle_{\SpHS\times\SpHS} ds\\
+&\int_0^t\frac{1}{2}\Tr_{{\SpHS}\times
{\SpHS}}\bigg[\bar{G}_{\textbf{YY}}(\textbf{Y}_s)
\Big(B(Y_s)Q^{1/2},B^\#(Y_s^\#)Q^{1/2}\Big)\\
&\hspace{1.5in}\Big(B(Y_s)Q^{1/2},B^\#(Y_s^\#)Q^{1/2}\Big)^\ast\bigg]
ds.
\end{align*}

Let us calculate the three integrands separately.
The first integrand is
\begin{align*}
\langle \bar{G}_{\textbf{Y}}(\textbf{Y}_s),
&\big(B(Y_s)dW_s,B^\#(Y_s^\#)dW_s\big)\rangle_{\SpHS\times\SpHS}\\
&=\Big(B(Y_s)dW_s\Big)(Y_s^\#+I)+(Y_s+I)\Big(B^\#(Y_s^\#)dW_s\Big)\\
&=(Y_s+I)dW_s(Y_s^\#+I)+(Y_s+I)dW_s^\#(Y_s^\#+I)
=0.
\end{align*}

The second integrand is
\begin{align*}
\langle \bar{G}_{\textbf{Y}}(\textbf{Y}_s),
&\big(F(Y_s),F^\#(Y_s^\#)\big)\rangle_{\SpHS\times\SpHS}\\
&=F(Y_s)(Y_s^\#+I)+(Y_s+I)F^\#(Y_s^\#)\\
&=\frac{1}{2}(Y_s+I)D(Y_s^\#+I)+\frac{1}{2}(Y_s+I)D^\#(Y_s^\#+I)\\
&=\frac{1}{2}(Y_s+I)(D+D^\#)(Y_s^\#+I)\\
&=(Y_s+I)D(Y_s^\#+I),
\end{align*}
where we have used the fact that $D=D^\#$, since $D$ is a diagonal matrix
with all real entries.

The third integrand is
\begin{align*}
& \frac{1}{2}\Tr_{{\SpHS}\times {\SpHS}}
\\
& \left[\bar{G}_{\textbf{YY}}(\textbf{Y}_s)
\left(B(Y_s)Q^{1/2},B^\#(Y_s^\#)Q^{1/2}\right)
\left(B(Y_s)Q^{1/2},B^\#(Y_s^\#)Q^{1/2}\right)^\ast\right]
\\
&
=\frac{1}{2}\Tr_{{\SpHS}}\left[\left(B(Y_s)Q^{1/2},B^\#(Y_s^\#)Q^{1/2}\right)^\ast
\bar{G}_{\textbf{YY}}(\textbf{Y}_s)
\left(B(Y_s)Q^{1/2},B^\#(Y_s^\#)Q^{1/2}\right) \right]\\
&=\frac{1}{2}\sum_{\xi\in\mathcal{B}_{\SpHS}}G_{\textbf{YY}}(\textbf{Y}_s)
\bigg(\Big(B(Y_s)Q^{1/2}\xi,B^\#(Y_s^\#)Q^{1/2}\xi\Big)\\
&\hspace{2in}
\otimes\Big(B(Y_s)Q^{1/2}\xi,B^\#(Y_s^\#)Q^{1/2}\xi\Big)\bigg)\\
&=\sum_{\xi\in\mathcal{B}_{\SpHS}}
\Big(B(Y_s)Q^{1/2}\xi\Big)\Big(B^\#(Y_s^\#)Q^{1/2}\xi\Big)\\
&=\sum_{\xi\in\mathcal{B}_{\SpHS}}(Y_s+I)
\bigg(\big(Q^{1/2}\xi\big)\big(Q^{1/2}\xi\big)^\#\bigg)(Y_s^\#+I)\\
&=-(Y_s+I)D(Y_s^\#+I),
\end{align*}
where the second equality follows from Lemma \ref{Trace}, and the
last equality follows from Lemma \ref{Sum_xi}.

The above calculations show that the stochastic differential of $G$
is zero. So $G(\textbf{Y}_t)=G(\textbf{Y}_0)=\Lambda(I)$ for any
$t>0$ and any nonzero linear real bounded functional $\Lambda$ on
${\SpHS}\times {\SpHS}$. This means $(Y_t+I)(Y_t^\#+I)=I$ almost
surely for any $t>0$. Similarly we can show $(Y_t^\#+I)(Y_t+I)=I$
almost surely for any $t>0$. Therefore $Y_t+I\in \Sp(\infty)$ almost
surely for any $t>0$.
\end{proof}

\section{Ricci curvature of $\Sp(\infty)$}

In \cite{Gordina1, Gordina2}, Gordina computed the Ricci curvatures of the quotient
space $\Diff(S^1)/S^1$ and several Hilbert-Schmidt groups.
Using this method, I computed the Ricci curvatures of the symplectic group
$\Sp(\infty)$.
The result shows that for most of the directions the Ricci curvatures
of $\Sp(\infty)$ are not bounded from below.

Let $G$ be a finite dimensional Lie group, and $\mathfrak{g}$ its Lie algebra.
Let $\langle\cdot,\cdot\rangle_\mathfrak{g}$ be an inner product on $\mathfrak{g}$.
Then $\langle\cdot,\cdot\rangle_\mathfrak{g}$ defines a unique left-invariant metric
on the Lie group $G$ compactible with the Lie group structure.
In \cite{Milnor1}, J. Milnor studied the Riemannian geometry of Lie groups.
For $x,y,z\in\mathfrak{g}$, the \emph{Levi-Civita connection} $\nabla_x$ is given by
\begin{equation}\label{Levi-Civita}
\langle\nabla_x y,z\rangle_\mathfrak{g}=\frac{1}{2}(\langle[x,y],z\rangle_\mathfrak{g}
-\langle[y,z],x\rangle_\mathfrak{g}+\langle[z,x],y\rangle_\mathfrak{g})
\end{equation}
The \emph{Riemann curvature tensor} $R_{xy}$ is given by
\begin{equation}\label{RTensor}
R_{xy}=\nabla_{[x,y]}-\nabla_x\nabla_y+\nabla_y\nabla_x
\end{equation}
For any orthogonal $x,y\in\mathfrak{g}$, the \emph{sectional curvature} $K(x,y)$ is given by
\begin{equation}\label{SectCurv}
K(x,y) = \langle R_{xy}(x),y\rangle_\mathfrak{g}
\end{equation}
Let us choose an orthonormal basis $\{\xi_i\}_{i=1}^N$ of $\mathfrak{g}$, where $N$ is the
dimension of the Lie group $G$. Let $x\in\mathfrak{g}$. Then the \emph{Ricci curvature}
$\Ric(x)$ is given by
\begin{equation}\label{RicciCurv}
\Ric(x) = \sum_{i=1}^N K(x,\xi_i) = \sum_{i=1}^N \langle R_{x\xi_i}(x),\xi_i\rangle_\mathfrak{g}
\end{equation}

The group $\Sp(\infty)$ and its Lie algebra $\fraksp(\infty)$ are defined in
Section 3 and Section 6.
Basically, elements in both the Lie group $\Sp(\infty)$
and the Lie algebra $\fraksp(\infty)$ are block matrices of the form:
\[
\left(
\begin{array}{ll}
a &b\\
\bar{b} &\bar{a}
\end{array}
\right)
\]
where each of the blocks is an infinite-dimensional matrix.
The blocks $a$ and $\bar{a}$ are complex conjugate with each other.
The blocks $b$ and $\bar{b}$ are also complex conjugate with each other,
are required to be a Hilbert-Schmidt matrices.
For a block matrix to be an element of $\Sp(\infty)$, it is also required that
the matrix is invertible, and preserve a certain symplectic form.
For a block matrix to be an element of $\fraksp(\infty)$, it is required that
$a+a^\dag=0$ or $a^T+\bar{a}=0$, which means the block $a$ is conjugate skew-symmetric,
and $b=b^T$, which means the block $b$ is symmetric.

To write the block matrix explicitly, we index the matrix by
$\mathbb{Z}\backslash\{0\}\times\mathbb{Z}\backslash\{0\}$,
so the matrix is written as $\{A_{mn}\}_{m,n\in\mathbb{Z}\backslash\{0\}}$.
An entry in block $a$ has $m,n>0$; an entry in block $\bar{a}$ has $m,n<0$;
an entry in block $b$ has $m>0,n<0$; an entry in block $\bar{b}$ has $m<0,n>0$.
The condition that blocks $a$ and $b$ are conjugate to blocks $\bar{a}$ and $\bar{b}$
can be expressed as $A_{m,n}=\overline{A_{-m,-n}}$.
The condition $a+a^\dag=0$ or $a^T+\bar{a}=0$ can be expressed as
$A_{nm}+\overline{A_{mn}}=0$ where $m,n>0$ or $m,n<0$.
The condition $b=b^T$ can be expressed as $A_{m,n}=A_{-n,-m}$ where $m>0,n<0$.

To find Ricci curvature, we need to choose a metric for the Lie algebra $\fraksp(\infty)$.
Let us define a sequence of positive numbers
\[
\{ \lambda_i\in\mathbb{R}_+ | \lambda_i=\lambda_{-i}, i\in\mathbb{Z}\backslash\{0\} \}
\]
The sequence $\{\lambda_i\}$ will serve as parameters to fine tune the metric
that we are going to choose.

\begin{remark}\label{CanonHSMetric}
Let us first consider the space $\HS$ of Hilbert-Schmidt matrices.
The Hilbert space $\HS$, if viewed as a complex Hilbert space, has a canonical
inner product given by:
\[
\langle A,B \rangle = \Tr(AB^\dag) = \Tr(A\bar{B}^T),
\hspace{.1in} A,B\in\HS
\]
If viewed as \emph{real} Hilbert space, $\HS$ has a canonical inner product given by:
for $A,B\in\HS$, writing $A=A_1+iA_2$ and $B=B_1+iB_2$, where $A_1,A_2,B_1,B_2$ are
matrices with real value entries, then
\[
\langle A,B \rangle = \Tr(A_1B_1^T)+\Tr(A_2B_2^T)
\]
Let $e_{ab}$ be the infinite-dimensional matrix with $1$ in the entry $(a,b)$,
and $0$ in all other entries, where $a, b$ are indices of the matrix such that
$a,b\in\mathbb{Z}\backslash\{0\}$.
Then the above canonical inner product on $\HS$ viewed as \emph{real} Hilbert space
is equivalent to choosing the set
\[
\{e_{ab},ie_{ab} | a,b\in\mathbb{Z}\backslash\{0\}\}
\]
as an orthonormal basis.
\end{remark}

\begin{definition}\label{HSMetric}
Let
\begin{equation}\label{HSxi}
\xi_{ab}=2\lambda_a\lambda_b e_{ab}
\end{equation}
We define an inner product $\langle\cdot,\cdot\rangle_{\HS}$ on $\HS$
by choosing the following set
\begin{equation}
\{ \xi_{ab},i\xi_{ab} | a,b\in\mathbb{Z}\backslash\{0\} \}
\end{equation}
as an orthonormal basis for the \emph{real} Hilbert space $\HS$.
\end{definition}

\begin{remark}\label{RecoverCanon}
If we set the parameter $\lambda_i=1/\sqrt{2}$ for all $i\in\mathbb{Z}\backslash\{0\}$
in Definition \ref{HSMetric}, we can recover the canonical inner product of $\HS$
(remark \ref{CanonHSMetric}) as a \emph{real} Hilbert space.
\end{remark}

The Lie algebra $\fraksp(\infty)$ may contain unbounded opertors.
For simplicity, we consider the subspace $\SpHS=\fraksp(\infty)\cap\HS$.
Now we can choose orthonormal set of the space $\SpHS$ according to the symmety
of matrices in the Lie algebra $\fraksp(\infty)$.

\begin{definition}
Let
\begin{eqnarray*}
&& \mu_{ab}^{Re}=\lambda_a\lambda_b (e_{a,b}-e_{b,a}+e_{-a,-b}-e_{-b,-a}),
\hspace{.1in} a > b > 0 \\&&
\mu_{ab}^{Im}=\lambda_a\lambda_b (ie_{a,b}-ie_{b,a}+ie_{-a,-b}-ie_{-b,-a}),
\hspace{.1in} a \ge b > 0 \\&&
\nu_{ab}^{Re}=\lambda_a\lambda_b (e_{a,b}+e_{-b,-a}+e_{-a,-b}+e_{b,a}),
\hspace{.1in} a \ge -b > 0 \\&&
\nu_{ab}^{Im}=\lambda_a\lambda_b (ie_{a,b}+ie_{-b,-a}-ie_{-a,-b}-ie_{b,a}),
\hspace{.1in} a \ge -b > 0
\end{eqnarray*}
Let
$A^{Re}=\{\mu_{ab}^{Re} | a > b > 0\}$,
$A^{Im}=\{\mu_{ab}^{Im} | a \ge b > 0\}$,
$B^{Re}=\{\nu_{ab}^{Re} | a \ge -b > 0\}$,
$B^{Im}=\{\nu_{ab}^{Im} | a \ge -b > 0\}$, and
$\mathcal{B}_\lambda = A^{Re} \cup A^{Im} \cup B^{Re} \cup B^{Im}$.
\end{definition}

\begin{remark}
It is easy to verify that matrices in the set $\mathcal{B}_\lambda$ all belong
to the space $\SpHS$. So $\mathcal{B}_\lambda$ is a subset of
$\SpHS$ and $\fraksp(\infty)$. Also, by definition of $\xi_{ab}$ (equation \ref{HSxi}),
it is easy to verify
\begin{eqnarray}\label{xicomb}
&&\mu_{ab}^{Re}=\frac{1}{2}
(\xi_{a,b}-\xi_{b,a}+\xi_{-a,-b}-\xi_{-b,-a}) \notag
\\&&
\mu_{ab}^{Im}=\frac{1}{2}
(i\xi_{a,b}-i\xi_{b,a}+i\xi_{-a,-b}-i\xi_{-b,-a})
\\&&
\nu_{ab}^{Re}=\frac{1}{2}
(\xi_{a,b}+\xi_{-b,-a}+\xi_{-a,-b}+\xi_{b,a}) \notag
\\&&
\nu_{ab}^{Im}=\frac{1}{2}
(i\xi_{a,b}+i\xi_{-b,-a}-i\xi_{-a,-b}-i\xi_{b,a}) \notag
\end{eqnarray}
\end{remark}

\begin{definition}\label{spmetric}
We define an inner product $\langle\cdot,\cdot\rangle_{\fraksp}$
on both $\SpHS$ and $\fraksp(\infty)$ by choosing
the set $\mathcal{B}_\lambda$ as an orthonormal set.
\end{definition}

\begin{remark}
We note that the inner product on $\SpHS$ and $\fraksp(\infty)$
is equivalent to the subspace inner product induced from the inner product
on $\HS$ defined in Definition \eqref{HSMetric}.
Therefore, for $x,y\in\SpHS$, $\langle x,y\rangle_\HS=\langle x,y\rangle_\fraksp$.
\end{remark}

\begin{remark}
For $\mu_{ab}^{Re}$, the indices satisfy $a>b>0$, which means
the entry is in the strict upper triangular block.
For $\mu_{ab}^{Im}$, the indices satisfy $a\ge b>0$, which means
the entry is in the upper triangular block including the diagonal.
For $\nu_{ab}^{Re}$, the indices satisfy $a\ge -b>0$, which means
the entry is in the other upper triangular block including the diagonal.
For $\nu_{ab}^{Im}$, the indices satisfy $a\ge -b>0$, which means
the entry is in the other upper triangular block including the diagonal.
\end{remark}

\begin{definition}\label{SpRicci}
Using Ricci curvature formula (Equation \ref{RicciCurv}) for $\Sp(\infty)$
and $\fraksp(\infty)$, we define, for $x\in\fraksp(\infty)$,
\begin{equation}
\Ric(x) = \sum_{\xi\in\mathcal{B}_\lambda} K(x,\xi)
= \sum_{\xi\in\mathcal{B}_\lambda} \langle R_{x\xi}(x),\xi\rangle_\fraksp
\end{equation}
By definition of $\mathcal{B}_\lambda$, the above sum will break into four parts:
\begin{equation}
\Ric(x)
=\sum_{a>b>0} K(x,\mu_{ab}^{Re})
+\sum_{a\ge b>0} K(x,\mu_{ab}^{Im})
+\sum_{a\ge -b>0} K(x,\nu_{ab}^{Re})
+\sum_{a\ge -b>0} K(x,\nu_{ab}^{Im})
\end{equation}
For computational reason, we define the following \emph{truncated Ricci curvature}:
\begin{align}
\Ric^N(x)
=\sum_{N\ge a>b>0} K(x,\mu_{ab}^{Re})
&+\sum_{N\ge a\ge b>0} K(x,\mu_{ab}^{Im})\notag\\
&+\sum_{N\ge a\ge -b>0} K(x,\nu_{ab}^{Re})
+\sum_{N\ge a\ge -b>0} K(x,\nu_{ab}^{Im})
\end{align}
We have $\Ric(x)=\lim_{N\to\infty} \Ric^N(x)$.
\end{definition}

In the rest of the paper, we will compute the following Ricci curvatures via
the corresponding truncated Ricci curvatures:
\[
\Ric(\mu_{ab}^{Re}),
\Ric(\mu_{ab}^{Im}),
\Ric(\nu_{ab}^{Re}),
\Ric(\nu_{ab}^{Im})
\]
All of these computations boil down to matrix multiplications.
The following lemma is an important tool to the computation of Ricci curvature.

\begin{lemma}\label{ConnFormula}
We have the following Levi-Civita connection formula, where $\delta$
is the Kronecker delta:
\begin{eqnarray*}
&&\nabla_{\xi_{ab}}\xi_{cd}
=\delta_{bc}\lambda_c^2\xi_{ad}-\delta_{da}\lambda_a^2\xi_{cb}
-\delta_{ca}\lambda_d^2\xi_{db}+\delta_{db}\lambda_c^2\xi_{ac}
+\delta_{bd}\lambda_a^2\xi_{ca}-\delta_{ac}\lambda_b^2\xi_{bd}
\\&&
\nabla_{i\xi_{ab}}i\xi_{cd}
=-\delta_{bc}\lambda_c^2\xi_{ad}+\delta_{da}\lambda_a^2\xi_{cb}
-\delta_{ca}\lambda_d^2\xi_{db}+\delta_{db}\lambda_c^2\xi_{ac}
+\delta_{bd}\lambda_a^2\xi_{ca}-\delta_{ac}\lambda_b^2\xi_{bd}
\\&&
\nabla_{\xi_{ab}}i\xi_{cd}
=\delta_{bc}\lambda_c^2i\xi_{ad}-\delta_{da}\lambda_a^2i\xi_{cb}
+\delta_{ca}\lambda_d^2i\xi_{db}-\delta_{db}\lambda_c^2i\xi_{ac}
+\delta_{bd}\lambda_a^2i\xi_{ca}-\delta_{ac}\lambda_b^2i\xi_{bd}
\\&&
\nabla_{i\xi_{ab}}\xi_{cd}
=\delta_{bc}\lambda_c^2i\xi_{ad}-\delta_{da}\lambda_a^2i\xi_{cb}
-\delta_{ca}\lambda_d^2i\xi_{db}+\delta_{db}\lambda_c^2i\xi_{ac}
-\delta_{bd}\lambda_a^2i\xi_{ca}+\delta_{ac}\lambda_b^2i\xi_{bd}
\end{eqnarray*}
\end{lemma}

\begin{proof}
We have
\[ \xi_{ab}\xi_{cd}=2\lambda_b^2\delta_{cb}\xi_{ad} \]
So
\[
[\xi_{ab},\xi_{cd}]=\xi_{ab}\xi_{cd}-\xi_{cd}\xi_{ab}
=2\lambda_c^2\delta_{cb}\xi_{ad}-2\lambda_a^2\delta_{ad}\xi_{cb}
\]

In the following, $\langle\cdot,\cdot\rangle$ stands for $\langle\cdot,\cdot\rangle_\HS$.
Using orthonormality,
\begin{eqnarray*}
&&
2\langle\nabla_{\xi_{ab}}\xi_{cd},\xi_{ef}\rangle
\\&&\hspace{.3in}
=\langle[\xi_{ab},\xi_{cd}],\xi_{ef}\rangle
-\langle[\xi_{cd},\xi_{ef}],\xi_{ab}\rangle+\langle[\xi_{ef},\xi_{ab}],\xi_{cd}\rangle
\\&&\hspace{.3in}
=2\delta_{bc}\delta_{ae}\delta_{df}\lambda_c^2
-2\delta_{da}\delta_{ce}\delta_{bf}\lambda_a^2
-2\delta_{de}\delta_{ca}\delta_{fb}\lambda_e^2
\\&&\hspace{.3in}
\hspace{.3in}+2\delta_{fc}\delta_{ea}\delta_{db}\lambda_c^2
+2\delta_{fa}\delta_{ec}\delta_{bd}\lambda_a^2
-2\delta_{be}\delta_{ac}\delta_{fd}\lambda_e^2
\end{eqnarray*}
and
\[
2\langle\nabla_{\xi_{ab}}\xi_{cd},i\xi_{ef}\rangle=\langle[\xi_{ab},\xi_{cd}],i\xi_{ef}\rangle
-\langle[\xi_{cd},i\xi_{ef}],\xi_{ab}\rangle+\langle[i\xi_{ef},\xi_{ab}],\xi_{cd}\rangle=0
\]
Therefore,
\[
\nabla_{\xi_{ab}}\xi_{cd}
=\delta_{bc}\lambda_c^2\xi_{ad}-\delta_{da}\lambda_a^2\xi_{cb}
-\delta_{ca}\lambda_d^2\xi_{db}+\delta_{db}\lambda_c^2\xi_{ac}
+\delta_{bd}\lambda_a^2\xi_{ca}-\delta_{ac}\lambda_b^2\xi_{bd}
\]
Similarly,
\begin{eqnarray*}
&&
2\langle\nabla_{i\xi_{ab}}i\xi_{cd},\xi_{ef}\rangle
\\&&\hspace{.3in}
=\langle[i\xi_{ab},i\xi_{cd}],\xi_{ef}\rangle
-\langle[i\xi_{cd},\xi_{ef}],i\xi_{ab}\rangle+\langle[\xi_{ef},i\xi_{ab}],i\xi_{cd}\rangle
\\&&\hspace{.3in}
=-\delta_{bc}\delta_{ae}\delta_{df}\lambda_c^2
+\delta_{da}\delta_{ce}\delta_{bf}\lambda_a^2
-\delta_{de}\delta_{ca}\delta_{fb}\lambda_e^2
\\&&\hspace{.3in}
\hspace{.2in}+\delta_{fc}\delta_{ea}\delta_{db}\lambda_c^2
+\delta_{fa}\delta_{ec}\delta_{bd}\lambda_a^2
-\delta_{be}\delta_{ac}\delta_{fd}\lambda_e^2
\end{eqnarray*}
and
\[
2\langle\nabla_{i\xi_{ab}}i\xi_{cd},i\xi_{ef}\rangle=\langle[i\xi_{ab},i\xi_{cd}],i\xi_{ef}\rangle
-\langle[i\xi_{cd},i\xi_{ef}],i\xi_{ab}\rangle+\langle[i\xi_{ef},i\xi_{ab}],i\xi_{cd}\rangle=0
\]
Therefore,
\[
\nabla_{i\xi_{ab}}i\xi_{cd}
=-\delta_{bc}\lambda_c^2\xi_{ad}+\delta_{da}\lambda_a^2\xi_{cb}
-\delta_{ca}\lambda_d^2\xi_{db}+\delta_{db}\lambda_c^2\xi_{ac}
+\delta_{bd}\lambda_a^2\xi_{ca}-\delta_{ac}\lambda_b^2\xi_{bd}
\]
Similarly,
\begin{eqnarray*}
&&
2\langle\nabla_{\xi_{ab}}i\xi_{cd},i\xi_{ef}\rangle
\\&&\hspace{.3in}
=\langle[\xi_{ab},i\xi_{cd}],i\xi_{ef}\rangle
-\langle[i\xi_{cd},i\xi_{ef}],\xi_{ab}\rangle+\langle[i\xi_{ef},\xi_{ab}],i\xi_{cd}\rangle
\\&&\hspace{.3in}
=\delta_{bc}\delta_{ae}\delta_{df}\lambda_c^2
-\delta_{da}\delta_{ce}\delta_{bf}\lambda_a^2
+\delta_{de}\delta_{ca}\delta_{fb}\lambda_e^2
\\&&\hspace{.3in}
\hspace{.2in}-\delta_{fc}\delta_{ea}\delta_{db}\lambda_c^2
+\delta_{fa}\delta_{ec}\delta_{bd}\lambda_a^2
-\delta_{be}\delta_{ac}\delta_{fd}\lambda_e^2
\end{eqnarray*}
and
\[
2\langle\nabla_{\xi_{ab}}i\xi_{cd},\xi_{ef}\rangle=\langle[\xi_{ab},i\xi_{cd}],\xi_{ef}\rangle
-\langle[i\xi_{cd},\xi_{ef}],\xi_{ab}\rangle+\langle[\xi_{ef},\xi_{ab}],i\xi_{cd}\rangle=0
\]
Therefore,
\[
\nabla_{\xi_{ab}}i\xi_{cd}
=\delta_{bc}\lambda_c^2i\xi_{ad}-\delta_{da}\lambda_a^2i\xi_{cb}
+\delta_{ca}\lambda_d^2i\xi_{db}-\delta_{db}\lambda_c^2i\xi_{ac}
+\delta_{bd}\lambda_a^2i\xi_{ca}-\delta_{ac}\lambda_b^2i\xi_{bd}
\]
Similarly,
\begin{eqnarray*}
&&
2\langle\nabla_{i\xi_{ab}}\xi_{cd},i\xi_{ef}\rangle
\\&&\hspace{.3in}
=\langle[i\xi_{ab},\xi_{cd}],i\xi_{ef}\rangle
-\langle[\xi_{cd},i\xi_{ef}],i\xi_{ab}\rangle+\langle[i\xi_{ef},i\xi_{ab}],\xi_{cd}\rangle
\\&&\hspace{.3in}
=\delta_{bc}\delta_{ae}\delta_{df}\lambda_c^2
-\delta_{da}\delta_{ce}\delta_{bf}\lambda_a^2
-\delta_{de}\delta_{ca}\delta_{fb}\lambda_e^2
\\&&\hspace{.3in}
\hspace{.2in}+\delta_{fc}\delta_{ea}\delta_{db}\lambda_c^2
-\delta_{fa}\delta_{ec}\delta_{bd}\lambda_a^2
+\delta_{be}\delta_{ac}\delta_{fd}\lambda_e^2
\end{eqnarray*}
and
\[
2\langle\nabla_{i\xi_{ab}}\xi_{cd},\xi_{ef}\rangle=\langle[i\xi_{ab},\xi_{cd}],\xi_{ef}\rangle
-\langle[\xi_{cd},\xi_{ef}],i\xi_{ab}\rangle+\langle[\xi_{ef},i\xi_{ab}],\xi_{cd}\rangle=0
\]
Therefore,
\[
\nabla_{i\xi_{ab}}\xi_{cd}
=\delta_{bc}\lambda_c^2i\xi_{ad}-\delta_{da}\lambda_a^2i\xi_{cb}
-\delta_{ca}\lambda_d^2i\xi_{db}+\delta_{db}\lambda_c^2i\xi_{ac}
-\delta_{bd}\lambda_a^2i\xi_{ca}+\delta_{ac}\lambda_b^2i\xi_{bd}
\]

\end{proof}

\begin{remark}
Once we have the above lemma,
we can use equation \eqref{xicomb} to change the basis elements
of $\fraksp(\infty)$ into the basis elements of $\HS$, and then use
formula \eqref{Levi-Civita}, \eqref{RTensor}, \eqref{SectCurv} and \eqref{RicciCurv}
to compute the Ricci curvature of $\Sp(\infty)$.  But since each basis
$\mu_{ab}^{Re}$, $\mu_{ab}^{Im}$, $\nu_{ab}^{Re}$, and $\nu_{ab}^{Im}$
has four terms, and each connection formula in the above lemma has six terms,
the combination will be huge. For example, the sectional curvature
\begin{align*}
K(\mu_{ab}^{Re},\mu_{cd}^{Re})
&=\langle R_{\mu_{ab}^{Re}\mu_{cd}^{Re}}(\mu_{ab}^{Re}),\mu_{cd}^{Re}\rangle\\
&=\langle \nabla_{[\mu_{ab}^{Re},\mu_{cd}^{Re}]}(\mu_{ab}^{Re})
-\nabla_{\mu_{ab}^{Re}}\nabla_{\mu_{cd}^{Re}}(\mu_{ab}^{Re})
+\nabla_{\mu_{cd}^{Re}}\nabla_{\mu_{ab}^{Re}}(\mu_{ab}^{Re}),
\mu_{cd}^{Re} \rangle
\end{align*}
will have $21,504$ terms.
Thereore, I use a computer program to facilitate the computation.
\end{remark}

\begin{theorem}
Let $a,b\in\mathbb{Z}\backslash\{0\}$.

For $a>b>0$,
\begin{align*}
\Ric^N(\mu_{ab}^{Re})
=\frac{1}{16}\Big[
&-24\lambda_a^4
-24\lambda_b^4
+48\lambda_a^2\lambda_b^2
-12\lambda_a^2\sum_{d=1}^{a-1}\lambda_d^2
+8\lambda_a^2\sum_{d=1}^{b-1}\lambda_d^2
+8\lambda_b^2\sum_{d=1}^{a-1}\lambda_d^2
\\
&-12\lambda_b^2\sum_{d=1}^{b-1}\lambda_d^2
+8\sum_{d=1}^{a-1}\lambda_d^4
+8\sum_{d=1}^{b-1}\lambda_d^4
-16N\lambda_a^4
-16N\lambda_b^4
-12\lambda_a^2\sum_{c=a+1}^{N}\lambda_c^2
\\
&+8\lambda_a^2\sum_{c=b+1}^{N}\lambda_c^2
+8\lambda_b^2\sum_{c=a+1}^{N}\lambda_c^2
-12\lambda_b^2\sum_{c=b+1}^{N}\lambda_c^2
+8\sum_{c=a+1}^{N}\lambda_c^4
+8\sum_{c=b+1}^{N}\lambda_c^4
\Big].
\end{align*}

For $a>b>0$,
\begin{align*}
\Ric^N(\mu_{ab}^{Im})
=\frac{1}{16}\Big[
&-40\lambda_a^4
-40\lambda_b^4
-32\lambda_a^2\lambda_b^2
-12\lambda_a^2\sum_{d=1}^{a-1}\lambda_d^2
-8\lambda_a^2\sum_{d=1}^{b-1}\lambda_d^2
-8\lambda_b^2\sum_{d=1}^{a-1}\lambda_d^2
\\
&-12\lambda_b^2\sum_{d=1}^{b-1}\lambda_d^2
+8\sum_{d=1}^{a-1}\lambda_d^4
+8\sum_{d=1}^{b-1}\lambda_d^4
-16N\lambda_a^4
-16N\lambda_b^4
-12\lambda_a^2\sum_{c=a+1}^{N}\lambda_c^2
\\
&-8\lambda_a^2\sum_{c=b+1}^{N}\lambda_c^2
-8\lambda_b^2\sum_{c=a+1}^{N}\lambda_c^2
-12\lambda_b^2\sum_{c=b+1}^{N}\lambda_c^2
+8\sum_{c=a+1}^{N}\lambda_c^4
+8\sum_{c=b+1}^{N}\lambda_c^4
\Big].
\end{align*}

For $a=b>0$,
\begin{align*}
\Ric^N(\mu_{ab}^{Im})=0.
\end{align*}

For $a>-b>0$,
\begin{align*}
\Ric^N(\nu_{ab}^{Re})
=\frac{1}{16}\Big[
&-40\lambda_a^4
-40\lambda_b^4
-48\lambda_a^2\lambda_b^2
-12\lambda_a^2\sum_{d=1}^{a-1}\lambda_d^2
-8\lambda_a^2\sum_{d=1}^{b-1}\lambda_d^2
-8\lambda_b^2\sum_{d=1}^{a-1}\lambda_d^2
\\
&-12\lambda_b^2\sum_{d=1}^{b-1}\lambda_d^2
+8\sum_{d=1}^{a-1}\lambda_d^4
+8\sum_{d=1}^{b-1}\lambda_d^4
-16N\lambda_a^4
-16N\lambda_b^4
-12\lambda_a^2\sum_{c=a+1}^{N}\lambda_c^2
\\
&-8\lambda_a^2\sum_{c=b+1}^{N}\lambda_c^2
-8\lambda_b^2\sum_{c=a+1}^{N}\lambda_c^2
-12\lambda_b^2\sum_{c=b+1}^{N}\lambda_c^2
+8\sum_{c=a+1}^{N}\lambda_c^4
+8\sum_{c=b+1}^{N}\lambda_c^4
\Big].
\end{align*}

For $a=-b>0$,
\begin{align*}
\Ric^N(\nu_{ab}^{Re})
=\frac{1}{16}\Big[
-192\lambda_a^4
-32\sum_{d=1}^{a-1}\lambda_d^4
-192N\lambda_a^4
-32\sum_{c=a+1}^{N}\lambda_c^4
\Big].
\end{align*}

For $a>-b>0$,
\begin{align*}
R^N(\nu_{ab}^{Im})
=\frac{1}{16}\Big[
&-40\lambda_a^4
-40\lambda_b^4
-32\lambda_a^2\lambda_b^2
-12\lambda_a^2\sum_{d=1}^{a-1}\lambda_d^2
-8\lambda_a^2\sum_{d=1}^{b-1}\lambda_d^2
-8\lambda_b^2\sum_{d=1}^{a-1}\lambda_d^2
\\
&-12\lambda_b^2\sum_{d=1}^{b-1}\lambda_d^2
+8\sum_{d=1}^{a-1}\lambda_d^4
+8\sum_{d=1}^{b-1}\lambda_d^4
-16N\lambda_a^4
-16N\lambda_b^4
-12\lambda_a^2\sum_{c=a+1}^{N}\lambda_c^2
\\
&-8\lambda_a^2\sum_{c=b+1}^{N}\lambda_c^2
-8\lambda_b^2\sum_{c=a+1}^{N}\lambda_c^2
-12\lambda_b^2\sum_{c=b+1}^{N}\lambda_c^2
+8\sum_{c=a+1}^{N}\lambda_c^4
+8\sum_{c=b+1}^{N}\lambda_c^4
\Big].
\end{align*}

For $a=-b>0$,
\begin{align*}
\Ric^N(\nu_{ab}^{Im}) =0.
\end{align*}

\end{theorem}

\begin{corollary}
If we set the parameter $\lambda_i=1/\sqrt{2}$,
for all $i\in\mathbb{Z}\backslash\{0\}$, then we recover the canonical
inner product on the space $\HS$ (remark \ref{RecoverCanon}). In this case, we have
\begin{align*}
&\Ric^N(\mu_{ab}^{Re})=-\frac{3}{8}N-\frac{1}{8},
& \textrm{for } a>b>0;\\
&\Ric^N(\mu_{ab}^{Im})=-\frac{7}{8}N-\frac{11}{8},
& \textrm{for } a>b>0;\\
&\Ric^N(\mu_{ab}^{Im})=0,
& \textrm{for } a=b>0;\\
&\Ric^N(\nu_{ab}^{Re})=-\frac{7}{8}N-\frac{13}{8},
& \textrm{for } a>-b>0;\\
&\Ric^N(\nu_{ab}^{Re})=-\frac{7}{2}N-\frac{5}{2},
& \textrm{for } a=-b>0;\\
&\Ric^N(\nu_{ab}^{Im})=-\frac{7}{8}N-\frac{11}{8},
& \textrm{for } a>-b>0;\\
&\Ric^N(\nu_{ab}^{Im})=0,
& \textrm{for } a=-b>0.
\end{align*}
\end{corollary}

\begin{remark}
By the above corollary, we see that for most of the basis element
$\xi\in\mathcal{B}_\lambda$, we have $\Ric(\xi)=\lim_{N\to\infty} \Ric^N(\xi)=-\infty$.
\end{remark}

\begin{proof} (of the theorem.)

The method of computing Ricci curvature and truncated Ricci curvature
is stated in Definition \ref{SpRicci}.
Ricci curvature is defined in terms of sectional curvature,
which can be expressed in terms of Riemann tensor and the inner product of the Lie algebra.
Riemann tensor is defined in terms of Levi-Civita connection.
The formula of Levi-Civita connection is the content of Lemma \ref{ConnFormula}.
So the method of computing Ricci curvature is straightforward.
But there are huge number of terms.
Therefore, I used a computer program to facilitate the computation.

\begin{align*}
&\Ric^N(\mu_{ab}^{Re})\\
&=\sum_{N\ge c>d>0}K(\mu_{ab}^{Re},\mu_{cd}^{Re})
+\sum_{N\ge c\ge d>0}K(\mu_{ab}^{Re},\mu_{cd}^{Im})
\\
&\hspace{2in}
+\sum_{N\ge c\ge -d>0}K(\mu_{ab}^{Re},\nu_{cd}^{Re})
+\sum_{N\ge c\ge -d>0}K(\mu_{ab}^{Re},\nu_{cd}^{Im})
\\
&=\sum_{N\ge c>d>0}K(\mu_{ab}^{Re},\mu_{cd}^{Re})
+\sum_{N\ge c\ge d>0}K(\mu_{ab}^{Re},\mu_{cd}^{Im})
\\
&\hspace{2in}
+\sum_{N\ge c\ge d>0}K(\mu_{ab}^{Re},\nu_{c,-d}^{Re})
+\sum_{N\ge c\ge d>0}K(\mu_{ab}^{Re},\nu_{c,-d}^{Im})
\\
&=\sum_{N\ge c>d>0}
\Big[K(\mu_{ab}^{Re},\mu_{cd}^{Re})+K(\mu_{ab}^{Re},\mu_{cd}^{Im})
+K(\mu_{ab}^{Re},\nu_{c,-d}^{Re})+K(\mu_{ab}^{Re},\nu_{c,-d}^{Im})\Big]
\\
&\hspace{.5in}+\sum_{N\ge c=d>0}
\Big[K(\mu_{ab}^{Re},\mu_{cd}^{Re})+K(\mu_{ab}^{Re},\mu_{cd}^{Im})
+K(\mu_{ab}^{Re},\nu_{c,-d}^{Re})+K(\mu_{ab}^{Re},\nu_{c,-d}^{Im})\Big]
\\
&:=\sum_{N\ge c>d>0}A_{upper}+\sum_{N\ge c=d>0}A_{diagonal}
\end{align*}
We have
\begin{eqnarray*}
&&A_{upper}=
\frac{1}{16}\Big[
-16\delta_{a,c}\lambda_a^4
-24\delta_{a,c}\delta_{a,d}\lambda_a^4
+24\delta_{a,c}\delta_{b,d}\lambda_a^4
-16\delta_{a,d}\lambda_a^4
\\&&\hspace{.3in}
\hspace{.3in}
+24\delta_{a,d}\delta_{b,c}\lambda_a^4
+8\delta_{a,c}\delta_{a,d}\lambda_a^2\lambda_b^2
+8\delta_{b,c}\delta_{b,d}\lambda_a^2\lambda_b^2
-12\delta_{a,d}\lambda_a^2\lambda_c^2
\\&&\hspace{.3in}
\hspace{.3in}
+8\delta_{b,d}\lambda_a^2\lambda_c^2
-12\delta_{a,c}\lambda_a^2\lambda_d^2
+8\delta_{b,c}\lambda_a^2\lambda_d^2
+24\delta_{a,c}\delta_{b,d}\lambda_b^4
+24\delta_{a,d}\delta_{b,c}\lambda_b^4
\\&&\hspace{.3in}
\hspace{.3in}
-16\delta_{b,c}\lambda_b^4
-24\delta_{b,c}\delta_{b,d}\lambda_b^4
-16\delta_{b,d}\lambda_b^4
+8\delta_{a,d}\lambda_b^2\lambda_c^2
-12\delta_{b,d}\lambda_b^2\lambda_c^2
\\&&\hspace{.3in}
\hspace{.3in}
+8\delta_{a,c}\lambda_b^2\lambda_d^2
-12\delta_{b,c}\lambda_b^2\lambda_d^2
+8\delta_{a,d}\lambda_c^4
+8\delta_{b,d}\lambda_c^4
+8\delta_{a,c}\lambda_d^4
+8\delta_{b,c}\lambda_d^4
\Big]
\end{eqnarray*}
and
\begin{eqnarray*}
&&A_{diagonal}=
\frac{1}{16}\Big[
-12\delta_{a,c}\lambda_a^4
-16\delta_{a,c}\delta_{a,d}\lambda_a^4
+18\delta_{a,c}\delta_{b,d}\lambda_a^4
-12\delta_{a,d}\lambda_a^4
\\&&\hspace{.3in}
\hspace{.3in}
+18\delta_{a,d}\delta_{b,c}\lambda_a^4
+12\delta_{a,c}\delta_{b,d}\lambda_a^2\lambda_b^2
+12\delta_{a,d}\delta_{b,c}\lambda_a^2\lambda_b^2
-18\delta_{a,d}\lambda_a^2\lambda_c^2
\\&&\hspace{.3in}
\hspace{.3in}
+12\delta_{b,d}\lambda_a^2\lambda_c^2
-18\delta_{a,c}\lambda_a^2\lambda_d^2
+12\delta_{b,c}\lambda_a^2\lambda_d^2
+18\delta_{a,c}\delta_{b,d}\lambda_b^4
+18\delta_{a,d}\delta_{b,c}\lambda_b^4
\\&&\hspace{.3in}
\hspace{.3in}
-12\delta_{b,c}\lambda_b^4
-16\delta_{b,c}\delta_{b,d}\lambda_b^4
-12\delta_{b,d}\lambda_b^4
+12\delta_{a,d}\lambda_b^2\lambda_c^2
-18\delta_{b,d}\lambda_b^2\lambda_c^2
\\&&\hspace{.3in}
\hspace{.3in}
+12\delta_{a,c}\lambda_b^2\lambda_d^2
-18\delta_{b,c}\lambda_b^2\lambda_d^2
+6\delta_{a,d}\lambda_c^4
+6\delta_{b,d}\lambda_c^4
+6\delta_{a,c}\lambda_d^4
+6\delta_{b,c}\lambda_d^4
\Big]
\end{eqnarray*}
So
\begin{eqnarray*}
&&\sum_{N\ge c>d>0}A_{upper}=
\frac{1}{16}\Big[
-16(a-1)\lambda_a^4
+24\lambda_a^4
-16(N-a)\lambda_a^4
-12\lambda_a^2\sum_{c=a+1}^{N}\lambda_c^2
\\&&\hspace{.3in}
\hspace{.3in}
+8\lambda_a^2\sum_{c=b+1}^{N}\lambda_c^2
-12\lambda_a^2\sum_{d=1}^{a-1}\lambda_d^2
+8\lambda_a^2\sum_{d=1}^{b-1}\lambda_d^2
+24\lambda_b^4
\\&&\hspace{.3in}
\hspace{.3in}
-16(b-1)\lambda_b^4
-16(N-b)\lambda_b^4
+8\lambda_b^2\sum_{c=a+1}^{N}\lambda_c^2
-12\lambda_b^2\sum_{c=b+1}^{N}\lambda_c^2
+8\lambda_b^2\sum_{d=1}^{a-1}\lambda_d^2
\\&&\hspace{.3in}
\hspace{.3in}
-12\lambda_b^2\sum_{d=1}^{b-1}\lambda_d^2
+8\sum_{c=a+1}^{N}\lambda_c^4
+8\sum_{c=b+1}^{N}\lambda_c^4
+8\sum_{d=1}^{a-1}\lambda_d^4
+8\sum_{d=1}^{b-1}\lambda_d^4
\Big]
\end{eqnarray*}
and
\begin{eqnarray*}
&&\sum_{N\ge c=d>0}A_{diagonal}=
\frac{1}{16}\Big[
-12\lambda_a^4
-16\lambda_a^4
-12\lambda_a^4
-18\lambda_a^4
+12\lambda_a^2\lambda_b^2
-18\lambda_a^4
+12\lambda_a^2\lambda_b^2
-12\lambda_b^4
\\&&\hspace{.3in}
\hspace{.3in}
-16\lambda_b^4
-12\lambda_b^4
+12\lambda_a^2\lambda_b^2
-18\lambda_b^4
+12\lambda_a^2\lambda_b^2
-18\lambda_b^4
+6\lambda_a^4
+6\lambda_b^4
+6\lambda_a^4
+6\lambda_b^4
\Big]
\end{eqnarray*}
Therefore, for $a>b>0$,
\begin{eqnarray*}
&&\Ric^N(\mu_{ab}^{Re})
=\sum_{N\ge c>d>0}A_{upper}+\sum_{N\ge c=d>0}A_{diagonal}
\\&&\hspace{.3in}
=\frac{1}{16}\Big[
-24\lambda_a^4
-24\lambda_b^4
+48\lambda_a^2\lambda_b^2
-12\lambda_a^2\sum_{d=1}^{a-1}\lambda_d^2
+8\lambda_a^2\sum_{d=1}^{b-1}\lambda_d^2
+8\lambda_b^2\sum_{d=1}^{a-1}\lambda_d^2
\\&&\hspace{.3in}
\hspace{.3in}
-12\lambda_b^2\sum_{d=1}^{b-1}\lambda_d^2
+8\sum_{d=1}^{a-1}\lambda_d^4
+8\sum_{d=1}^{b-1}\lambda_d^4
-16N\lambda_a^4
-16N\lambda_b^4
-12\lambda_a^2\sum_{c=a+1}^{N}\lambda_c^2
\\&&\hspace{.3in}
\hspace{.3in}
+8\lambda_a^2\sum_{c=b+1}^{N}\lambda_c^2
+8\lambda_b^2\sum_{c=a+1}^{N}\lambda_c^2
-12\lambda_b^2\sum_{c=b+1}^{N}\lambda_c^2
+8\sum_{c=a+1}^{N}\lambda_c^4
+8\sum_{c=b+1}^{N}\lambda_c^4
\Big]
\end{eqnarray*}

Next,
\begin{align*}
&\Ric^N(\mu_{ab}^{Im})\\
&=\sum_{N\ge c>d>0}K(\mu_{ab}^{Im},\mu_{cd}^{Re})
+\sum_{N\ge c\ge d>0}K(\mu_{ab}^{Im},\mu_{cd}^{Im})
\\
&\hspace{2in}
+\sum_{N\ge c\ge -d>0}K(\mu_{ab}^{Im},\nu_{cd}^{Re})
+\sum_{N\ge c\ge -d>0}K(\mu_{ab}^{Im},\nu_{cd}^{Im})
\\
&=\sum_{N\ge c>d>0}K(\mu_{ab}^{Im},\mu_{cd}^{Re})
+\sum_{N\ge c\ge d>0}K(\mu_{ab}^{Im},\mu_{cd}^{Im})
\\
&\hspace{2in}
+\sum_{N\ge c\ge d>0}K(\mu_{ab}^{Im},\nu_{c,-d}^{Re})
+\sum_{N\ge c\ge d>0}K(\mu_{ab}^{Im},\nu_{c,-d}^{Im})
\\
&=\sum_{N\ge c>d>0}
\Big[K(\mu_{ab}^{Im},\mu_{cd}^{Re})+K(\mu_{ab}^{Im},\mu_{cd}^{Im})
+K(\mu_{ab}^{Im},\nu_{c,-d}^{Re})+K(\mu_{ab}^{Im},\nu_{c,-d}^{Im})\Big]
\\
&\hspace{.5in}+\sum_{N\ge c=d>0}
\Big[K(\mu_{ab}^{Im},\mu_{cd}^{Re})+K(\mu_{ab}^{Im},\mu_{cd}^{Im})
+K(\mu_{ab}^{Im},\nu_{c,-d}^{Re})+K(\mu_{ab}^{Im},\nu_{c,-d}^{Im})\Big]
\\
&:=\sum_{N\ge c>d>0}B_{upper}+\sum_{N\ge c=d>0}B_{diagonal}
\end{align*}
We have
\begin{eqnarray*}
&&B_{upper}=
\frac{1}{16}\Big[
+32\delta_{a,b}\delta_{a,c}\lambda_a^4
+32\delta_{a,b}\delta_{a,d}\lambda_a^4
-16\delta_{a,c}\lambda_a^4
-24\delta_{a,c}\delta_{a,d}\lambda_a^4
+8\delta_{a,c}\delta_{b,d}\lambda_a^4
\\&&\hspace{.3in}
\hspace{.3in}
-16\delta_{a,d}\lambda_a^4
+8\delta_{a,d}\delta_{b,c}\lambda_a^4
-8\delta_{a,c}\delta_{a,d}\lambda_a^2\lambda_b^2
+16\delta_{a,c}\delta_{b,d}\lambda_a^2\lambda_b^2
+16\delta_{a,d}\delta_{b,c}\lambda_a^2\lambda_b^2
\\&&\hspace{.3in}
\hspace{.3in}
-8\delta_{b,c}\delta_{b,d}\lambda_a^2\lambda_b^2
+40\delta_{a,b}\delta_{a,d}\lambda_a^2\lambda_c^2
-12\delta_{a,d}\lambda_a^2\lambda_c^2
-8\delta_{b,d}\lambda_a^2\lambda_c^2
+40\delta_{a,b}\delta_{a,c}\lambda_a^2\lambda_d^2
\\&&\hspace{.3in}
\hspace{.3in}
-12\delta_{a,c}\lambda_a^2\lambda_d^2
-8\delta_{b,c}\lambda_a^2\lambda_d^2
+8\delta_{a,c}\delta_{b,d}\lambda_b^4
+8\delta_{a,d}\delta_{b,c}\lambda_b^4
-16\delta_{b,c}\lambda_b^4
-24\delta_{b,c}\delta_{b,d}\lambda_b^4
\\&&\hspace{.3in}
\hspace{.3in}
-16\delta_{b,d}\lambda_b^4
-8\delta_{a,d}\lambda_b^2\lambda_c^2
-12\delta_{b,d}\lambda_b^2\lambda_c^2
-8\delta_{a,c}\lambda_b^2\lambda_d^2
-12\delta_{b,c}\lambda_b^2\lambda_d^2
\\&&\hspace{.3in}
\hspace{.3in}
-16\delta_{a,b}\delta_{a,d}\lambda_c^4
+8\delta_{a,d}\lambda_c^4
+8\delta_{b,d}\lambda_c^4
-16\delta_{a,b}\delta_{a,c}\lambda_d^4
+8\delta_{a,c}\lambda_d^4
+8\delta_{b,c}\lambda_d^4
\Big]
\end{eqnarray*}
and
\begin{eqnarray*}
&&B_{diagonal}=
\frac{1}{16}\Big[
+24\delta_{a,b}\delta_{a,c}\lambda_a^4
-32\delta_{a,b}\delta_{a,c}\delta_{a,d}\lambda_a^4
+24\delta_{a,b}\delta_{a,d}\lambda_a^4
-12\delta_{a,c}\lambda_a^4
-16\delta_{a,c}\delta_{a,d}\lambda_a^4
\\&&\hspace{.3in}
\hspace{.3in}
+6\delta_{a,c}\delta_{b,d}\lambda_a^4
-12\delta_{a,d}\lambda_a^4
+6\delta_{a,d}\delta_{b,c}\lambda_a^4
+20\delta_{a,c}\delta_{b,d}\lambda_a^2\lambda_b^2
+20\delta_{a,d}\delta_{b,c}\lambda_a^2\lambda_b^2
\\&&\hspace{.3in}
\hspace{.3in}
+60\delta_{a,b}\delta_{a,d}\lambda_a^2\lambda_c^2
-18\delta_{a,d}\lambda_a^2\lambda_c^2
-12\delta_{b,d}\lambda_a^2\lambda_c^2
+60\delta_{a,b}\delta_{a,c}\lambda_a^2\lambda_d^2
-18\delta_{a,c}\lambda_a^2\lambda_d^2
\\&&\hspace{.3in}
\hspace{.3in}
-12\delta_{b,c}\lambda_a^2\lambda_d^2
+6\delta_{a,c}\delta_{b,d}\lambda_b^4
+6\delta_{a,d}\delta_{b,c}\lambda_b^4
-12\delta_{b,c}\lambda_b^4
-16\delta_{b,c}\delta_{b,d}\lambda_b^4
\\&&\hspace{.3in}
\hspace{.3in}
-12\delta_{b,d}\lambda_b^4
-12\delta_{a,d}\lambda_b^2\lambda_c^2
-18\delta_{b,d}\lambda_b^2\lambda_c^2
-12\delta_{a,c}\lambda_b^2\lambda_d^2
-18\delta_{b,c}\lambda_b^2\lambda_d^2
\\&&\hspace{.3in}
\hspace{.3in}
-12\delta_{a,b}\delta_{a,d}\lambda_c^4
+6\delta_{a,d}\lambda_c^4
+6\delta_{b,d}\lambda_c^4
-12\delta_{a,b}\delta_{a,c}\lambda_d^4
+6\delta_{a,c}\lambda_d^4
+6\delta_{b,c}\lambda_d^4
\Big]
\end{eqnarray*}
For $a>b>0$,
\begin{eqnarray*}
&&\sum_{N\ge c>d>0}B_{upper}=
\frac{1}{16}\Big[
-16(a-1)\lambda_a^4
+8\lambda_a^4
-16(N-a)\lambda_a^4
+16\lambda_a^2\lambda_b^2
-12\lambda_a^2\sum_{c=a+1}^{N}\lambda_c^2
\\&&\hspace{.3in}
\hspace{.3in}
-8\lambda_a^2\sum_{c=b+1}^{N}\lambda_c^2
-12\lambda_a^2\sum_{d=1}^{a-1}\lambda_d^2
-8\lambda_a^2\sum_{d=1}^{b-1}\lambda_d^2
+8\lambda_b^4
-16(b-1)\lambda_b^4
\\&&\hspace{.3in}
\hspace{.3in}
-16(N-b)\lambda_b^4
-8\lambda_b^2\sum_{c=a+1}^{N}\lambda_c^2
-12\lambda_b^2\sum_{c=b+1}^{N}\lambda_c^2
-8\lambda_b^2\sum_{d=1}^{a-1}\lambda_d^2
-12\lambda_b^2\sum_{d=1}^{b-1}\lambda_d^2
\\&&\hspace{.3in}
\hspace{.3in}
+8\sum_{c=a+1}^{N}\lambda_c^4
+8\sum_{c=b+1}^{N}\lambda_c^4
+8\sum_{d=1}^{a-1}\lambda_d^4
+8\sum_{d=1}^{b-1}\lambda_d^4
\Big]
\end{eqnarray*}
For $a=b>0$,
\begin{eqnarray*}
&&\sum_{N\ge c>d>0}B_{upper}=
\frac{1}{16}\Big[
+32(a-1)\lambda_a^4
+32(N-a)\lambda_a^4
-16(a-1)\lambda_a^4
-16(N-a)\lambda_a^4
\\&& \hspace{.1in}
+40\lambda_a^2\sum_{c=a+1}^{N}\lambda_c^2
-12\lambda_a^2\sum_{c=a+1}^{N}\lambda_c^2
-8\lambda_a^2\sum_{c=a+1}^{N}\lambda_c^2
+40\lambda_a^2\sum_{d=1}^{a-1}\lambda_d^2
-12\lambda_a^2\sum_{d=1}^{a-1}\lambda_d^2
-8\lambda_a^2\sum_{d=1}^{a-1}\lambda_d^2
\\&& \hspace{.1in}
-16(a-1)\lambda_a^4
-16(N-a)\lambda_a^4
-8\lambda_a^2\sum_{c=a+1}^{N}\lambda_c^2
-12\lambda_a^2\sum_{c=a+1}^{N}\lambda_c^2
-8\lambda_a^2\sum_{d=1}^{a-1}\lambda_d^2
\\&& \hspace{.1in}
-12\lambda_a^2\sum_{d=1}^{a-1}\lambda_d^2
-16\sum_{c=a+1}^{N}\lambda_c^4
+8\sum_{c=a+1}^{N}\lambda_c^4
+8\sum_{c=a+1}^{N}\lambda_c^4
-16\sum_{d=1}^{a-1}\lambda_d^4
+8\sum_{d=1}^{a-1}\lambda_d^4
+8\sum_{d=1}^{a-1}\lambda_d^4
\Big]
\end{eqnarray*}
For $a>b>0$,
\begin{eqnarray*}
&&\sum_{N\ge c=d>0}B_{diagonal}=
\frac{1}{16}\Big[
-12\lambda_a^4
-16\lambda_a^4
-12\lambda_a^4
-18\lambda_a^4
-12\lambda_a^2\lambda_b^2
-18\lambda_a^4
-12\lambda_a^2\lambda_b^2
-12\lambda_b^4
\\&&\hspace{.3in}
\hspace{.3in}
-16\lambda_b^4
-12\lambda_b^4
-12\lambda_a^2\lambda_b^2
-18\lambda_b^4
-12\lambda_a^2\lambda_b^2
-18\lambda_b^4
+6\lambda_a^4
+6\lambda_b^4
+6\lambda_a^4
+6\lambda_b^4
\Big]
\end{eqnarray*}
For $a=b>0$,
\begin{eqnarray*}
\sum_{N\ge c=d>0}B_{diagonal}=0
\end{eqnarray*}
Therefore, for $a>b>0$,
\begin{eqnarray*}
&&\Ric^N(\mu_{ab}^{Im})
=\sum_{N\ge c>d>0}B_{upper}+\sum_{N\ge c=d>0}B_{diagonal}
\\&&\hspace{.3in}
=\frac{1}{16}\Big[
-40\lambda_a^4
-40\lambda_b^4
-32\lambda_a^2\lambda_b^2
-12\lambda_a^2\sum_{d=1}^{a-1}\lambda_d^2
-8\lambda_a^2\sum_{d=1}^{b-1}\lambda_d^2
-8\lambda_b^2\sum_{d=1}^{a-1}\lambda_d^2
\\&&\hspace{.3in}
\hspace{.3in}
-12\lambda_b^2\sum_{d=1}^{b-1}\lambda_d^2
+8\sum_{d=1}^{a-1}\lambda_d^4
+8\sum_{d=1}^{b-1}\lambda_d^4
-16N\lambda_a^4
-16N\lambda_b^4
-12\lambda_a^2\sum_{c=a+1}^{N}\lambda_c^2
\\&&\hspace{.3in}
\hspace{.3in}
-8\lambda_a^2\sum_{c=b+1}^{N}\lambda_c^2
-8\lambda_b^2\sum_{c=a+1}^{N}\lambda_c^2
-12\lambda_b^2\sum_{c=b+1}^{N}\lambda_c^2
+8\sum_{c=a+1}^{N}\lambda_c^4
+8\sum_{c=b+1}^{N}\lambda_c^4
\Big]
\end{eqnarray*}
and for $a=b>0$,
\begin{eqnarray*}
&&R^N(\mu_{ab}^{Im})
=\sum_{N\ge c>d>0}B_{upper}+\sum_{N\ge c=d>0}B_{diagonal}
=0
\end{eqnarray*}

Next, we compute $R^N(\nu_{ab}^{Re})$ for $a\ge-b>0$. Replacing $b$ with
$-b$, it's equivalent to computing $R^N(\nu_{a,-b}^{Re})$ for $a\ge b>0$.
\begin{align*}
&\Ric^N(\nu_{a,-b}^{Re})\\
&=\sum_{N\ge c>d>0}K(\nu_{a,-b}^{Re},\mu_{cd}^{Re})
+\sum_{N\ge c\ge d>0}K(\nu_{a,-b}^{Re},\mu_{cd}^{Im})\\
&\hspace{.5in}
+\sum_{N\ge c\ge -d>0}K(\nu_{a,-b}^{Re},\nu_{cd}^{Re})
+\sum_{N\ge c\ge -d>0}K(\nu_{a,-b}^{Re},\nu_{cd}^{Im})
\\
&=\sum_{N\ge c>d>0}K(\nu_{a,-b}^{Re},\mu_{cd}^{Re})
+\sum_{N\ge c\ge d>0}K(\nu_{a,-b}^{Re},\mu_{cd}^{Im})\\
&\hspace{.5in}
+\sum_{N\ge c\ge d>0}K(\nu_{a,-b}^{Re},\nu_{c,-d}^{Re})
+\sum_{N\ge c\ge d>0}K(\nu_{a,-b}^{Re},\nu_{c,-d}^{Im})
\\
&=\sum_{N\ge c>d>0}
\Big[K(\nu_{a,-b}^{Re},\mu_{cd}^{Re})+K(\nu_{a,-b}^{Re},\mu_{cd}^{Im})
+K(\nu_{a,-b}^{Re},\nu_{c,-d}^{Re})+K(\nu_{a,-b}^{Re},\nu_{c,-d}^{Im})\Big]
\\
&\hspace{.5in}+\sum_{N\ge c=d>0}
\Big[K(\nu_{a,-b}^{Re},\mu_{cd}^{Re})+K(\nu_{a,-b}^{Re},\mu_{cd}^{Im})
+K(\nu_{a,-b}^{Re},\nu_{c,-d}^{Re})+K(\nu_{a,-b}^{Re},\nu_{c,-d}^{Im})\Big]
\\
&:=\sum_{N\ge c>d>0}C_{upper}+\sum_{N\ge c=d>0}C_{diagonal}
\end{align*}
We have
\begin{eqnarray*}
&&C_{upper}=
\frac{1}{16}\Big[
-160\delta_{a,b}\delta_{a,c}\lambda_a^4
+480\delta_{a,b}\delta_{a,c}\delta_{a,d}\lambda_a^4
-160\delta_{a,b}\delta_{a,d}\lambda_a^4
-16\delta_{a,c}\lambda_a^4
\\&&\hspace{.3in}
-24\delta_{a,c}\delta_{a,d}\lambda_a^4
+8\delta_{a,c}\delta_{b,d}\lambda_a^4
-16\delta_{a,d}\lambda_a^4
+8\delta_{a,d}\delta_{b,c}\lambda_a^4
-8\delta_{a,c}\delta_{a,d}\lambda_a^2\lambda_b^2
\\&&\hspace{.3in}
-8\delta_{b,c}\delta_{b,d}\lambda_a^2\lambda_b^2
-24\delta_{a,b}\delta_{a,d}\lambda_a^2\lambda_c^2
-12\delta_{a,d}\lambda_a^2\lambda_c^2
-8\delta_{b,d}\lambda_a^2\lambda_c^2
\\&&\hspace{.3in}
-24\delta_{a,b}\delta_{a,c}\lambda_a^2\lambda_d^2
-12\delta_{a,c}\lambda_a^2\lambda_d^2
-8\delta_{b,c}\lambda_a^2\lambda_d^2
+8\delta_{a,c}\delta_{b,d}\lambda_b^4
+8\delta_{a,d}\delta_{b,c}\lambda_b^4
\\&&\hspace{.3in}
-16\delta_{b,c}\lambda_b^4
-24\delta_{b,c}\delta_{b,d}\lambda_b^4
-16\delta_{b,d}\lambda_b^4
-8\delta_{a,d}\lambda_b^2\lambda_c^2
-12\delta_{b,d}\lambda_b^2\lambda_c^2
-8\delta_{a,c}\lambda_b^2\lambda_d^2
\\&&\hspace{.3in}
-12\delta_{b,c}\lambda_b^2\lambda_d^2
+16\delta_{a,b}\delta_{a,d}\lambda_c^4
+8\delta_{a,d}\lambda_c^4
+8\delta_{b,d}\lambda_c^4
+16\delta_{a,b}\delta_{a,c}\lambda_d^4
+8\delta_{a,c}\lambda_d^4
+8\delta_{b,c}\lambda_d^4
\Big]
\end{eqnarray*}
and
\begin{eqnarray*}
&&C_{diagonal}=
\frac{1}{16}\Big[
-120\delta_{a,b}\delta_{a,c}\lambda_a^4
+480\delta_{a,b}\delta_{a,c}\delta_{a,d}\lambda_a^4
-120\delta_{a,b}\delta_{a,d}\lambda_a^4
-12\delta_{a,c}\lambda_a^4
\\&&\hspace{.3in}
\hspace{.3in}
-16\delta_{a,c}\delta_{a,d}\lambda_a^4
-6\delta_{a,c}\delta_{b,d}\lambda_a^4
-12\delta_{a,d}\lambda_a^4
-6\delta_{a,d}\delta_{b,c}\lambda_a^4
+4\delta_{a,c}\delta_{b,d}\lambda_a^2\lambda_b^2
\\&&\hspace{.3in}
\hspace{.3in}
+4\delta_{a,d}\delta_{b,c}\lambda_a^2\lambda_b^2
-36\delta_{a,b}\delta_{a,d}\lambda_a^2\lambda_c^2
-18\delta_{a,d}\lambda_a^2\lambda_c^2
-12\delta_{b,d}\lambda_a^2\lambda_c^2
-36\delta_{a,b}\delta_{a,c}\lambda_a^2\lambda_d^2
\\&&\hspace{.3in}
\hspace{.3in}
-18\delta_{a,c}\lambda_a^2\lambda_d^2
-12\delta_{b,c}\lambda_a^2\lambda_d^2
-6\delta_{a,c}\delta_{b,d}\lambda_b^4
-6\delta_{a,d}\delta_{b,c}\lambda_b^4
-12\delta_{b,c}\lambda_b^4
-16\delta_{b,c}\delta_{b,d}\lambda_b^4
\\&&\hspace{.3in}
\hspace{.3in}
-12\delta_{b,d}\lambda_b^4
-12\delta_{a,d}\lambda_b^2\lambda_c^2
-18\delta_{b,d}\lambda_b^2\lambda_c^2
-12\delta_{a,c}\lambda_b^2\lambda_d^2
-18\delta_{b,c}\lambda_b^2\lambda_d^2
\\&&\hspace{.3in}
\hspace{.3in}
+12\delta_{a,b}\delta_{a,d}\lambda_c^4
+6\delta_{a,d}\lambda_c^4
+6\delta_{b,d}\lambda_c^4
+12\delta_{a,b}\delta_{a,c}\lambda_d^4
+6\delta_{a,c}\lambda_d^4
+6\delta_{b,c}\lambda_d^4
\Big]
\end{eqnarray*}
For $a>b>0$,
\begin{eqnarray*}
&&\sum_{N\ge c>d>0}C_{upper}=
\frac{1}{16}\Big[
-16(a-1)\lambda_a^4
+8\lambda_a^4
-16(N-a)\lambda_a^4
-12\lambda_a^2\sum_{c=a+1}^{N}\lambda_c^2
\\&&\hspace{.3in}
\hspace{.3in}
-8\lambda_a^2\sum_{c=b+1}^{N}\lambda_c^2
-12\lambda_a^2\sum_{d=1}^{a-1}\lambda_d^2
-8\lambda_a^2\sum_{d=1}^{b-1}\lambda_d^2
+8\lambda_b^4
\\&&\hspace{.3in}
\hspace{.3in}
-16(b-1)\lambda_b^4
-16(N-b)\lambda_b^4
-8\lambda_b^2\sum_{c=a+1}^{N}\lambda_c^2
-12\lambda_b^2\sum_{c=b+1}^{N}\lambda_c^2
-8\lambda_b^2\sum_{d=1}^{a-1}\lambda_d^2
\\&&\hspace{.3in}
\hspace{.3in}
-12\lambda_b^2\sum_{d=1}^{b-1}\lambda_d^2
+8\sum_{c=a+1}^{N}\lambda_c^4
+8\sum_{c=b+1}^{N}\lambda_c^4
+8\sum_{d=1}^{a-1}\lambda_d^4
+8\sum_{d=1}^{b-1}\lambda_d^4
\Big]
\end{eqnarray*}
For $a=b>0$,
\begin{eqnarray*}
&&\sum_{N\ge c>d>0}C_{upper}=
\frac{1}{16}\Big[
-160(a-1)\lambda_a^4
-160(N-a)\lambda_a^4
-16(a-1)\lambda_a^4
-16(N-a)\lambda_a^4
\\&&\hspace{.3in}
\hspace{.3in}
-24\lambda_a^2\sum_{c=a+1}^{N}\lambda_c^2
-12\lambda_a^2\sum_{c=a+1}^{N}\lambda_c^2
-8\lambda_a^2\sum_{c=a+1}^{N}\lambda_c^2
-24\lambda_a^2\sum_{d=1}^{a-1}\lambda_d^2
-12\lambda_a^2\sum_{d=1}^{a-1}\lambda_d^2
\\&&\hspace{.3in}
\hspace{.3in}
-8\lambda_a^2\sum_{d=1}^{a-1}\lambda_d^2
-16(a-1)\lambda_a^4
-16(N-a)\lambda_a^4
-8\lambda_a^2\sum_{c=a+1}^{N}\lambda_c^2
\\&&\hspace{.3in}
\hspace{.3in}
-12\lambda_a^2\sum_{c=a+1}^{N}\lambda_c^2
-8\lambda_a^2\sum_{d=1}^{a-1}\lambda_d^2
-12\lambda_a^2\sum_{d=1}^{a-1}\lambda_d^2
+16\sum_{c=a+1}^{N}\lambda_c^4
+8\sum_{c=a+1}^{N}\lambda_c^4
\\&&\hspace{.3in}
\hspace{.3in}
+8\sum_{c=a+1}^{N}\lambda_c^4
+16\sum_{d=1}^{a-1}\lambda_d^4
+8\sum_{d=1}^{a-1}\lambda_d^4
+8\sum_{d=1}^{a-1}\lambda_d^4
\Big]
\end{eqnarray*}
For $a>b>0$,
\begin{eqnarray*}
&&\sum_{N\ge c=d>0}C_{diagonal}=
\frac{1}{16}\Big[
-12\lambda_a^4
-16\lambda_a^4
-12\lambda_a^4
-18\lambda_a^4
-12\lambda_a^2\lambda_b^2
-18\lambda_a^4
-12\lambda_a^2\lambda_b^2
-12\lambda_b^4
\\&&\hspace{.3in}
\hspace{.3in}
-16\lambda_b^4
-12\lambda_b^4
-12\lambda_a^2\lambda_b^2
-18\lambda_b^4
-12\lambda_a^2\lambda_b^2
-18\lambda_b^4
+6\lambda_a^4
+6\lambda_b^4
+6\lambda_a^4
+6\lambda_b^4
\Big]
\end{eqnarray*}
For $a=b>0$,
\begin{eqnarray*}
\sum_{N\ge c=d>0}C_{diagonal}=0
\end{eqnarray*}
Therefore, for $a>-b>0$,
\begin{eqnarray*}
&&\Ric^N(\nu_{ab}^{Re})
=\sum_{N\ge c>d>0}C_{upper}+\sum_{N\ge c=d>0}C_{diagonal}
\\&&\hspace{.3in}
=\frac{1}{16}\Big[
-40\lambda_a^4
-40\lambda_b^4
-48\lambda_a^2\lambda_b^2
-12\lambda_a^2\sum_{d=1}^{a-1}\lambda_d^2
-8\lambda_a^2\sum_{d=1}^{b-1}\lambda_d^2
-8\lambda_b^2\sum_{d=1}^{a-1}\lambda_d^2
\\&&\hspace{.3in}
\hspace{.3in}
-12\lambda_b^2\sum_{d=1}^{b-1}\lambda_d^2
+8\sum_{d=1}^{a-1}\lambda_d^4
+8\sum_{d=1}^{b-1}\lambda_d^4
-16N\lambda_a^4
-16N\lambda_b^4
-12\lambda_a^2\sum_{c=a+1}^{N}\lambda_c^2
\\&&\hspace{.3in}
\hspace{.3in}
-8\lambda_a^2\sum_{c=b+1}^{N}\lambda_c^2
-8\lambda_b^2\sum_{c=a+1}^{N}\lambda_c^2
-12\lambda_b^2\sum_{c=b+1}^{N}\lambda_c^2
+8\sum_{c=a+1}^{N}\lambda_c^4
+8\sum_{c=b+1}^{N}\lambda_c^4
\Big]
\end{eqnarray*}
and, for $a=-b>0$,
\begin{eqnarray*}
&&R^N(\nu_{ab}^{Re})
=\sum_{N\ge c>d>0}C_{upper}+\sum_{N\ge c=d>0}C_{diagonal}
\\&&\hspace{.3in}
=\frac{1}{16}\Big[
-192\lambda_a^4
-32\sum_{d=1}^{a-1}\lambda_d^4
-192N\lambda_a^4
-32\sum_{c=a+1}^{N}\lambda_c^4
\Big]
\end{eqnarray*}

Next, we compute $R^N(\nu_{ab}^{Im})$ for $a\ge-b>0$. Replacing $b$ with
$-b$, it's equivalent to computing $R^N(\nu_{a,-b}^{Im})$ for $a\ge b>0$.
\begin{align*}
&\Ric^N(\nu_{a,-b}^{Im})\\
&=\sum_{N\ge c>d>0}K(\nu_{a,-b}^{Im},\mu_{cd}^{Re})
+\sum_{N\ge c\ge d>0}K(\nu_{a,-b}^{Im},\mu_{cd}^{Im})\\
&\hspace{.5in}
+\sum_{N\ge c\ge -d>0}K(\nu_{a,-b}^{Im},\nu_{cd}^{Re})
+\sum_{N\ge c\ge -d>0}K(\nu_{a,-b}^{Im},\nu_{cd}^{Im})
\\
&=\sum_{N\ge c>d>0}K(\nu_{a,-b}^{Im},\mu_{cd}^{Re})
+\sum_{N\ge c\ge d>0}K(\nu_{a,-b}^{Im},\mu_{cd}^{Im})\\
&\hspace{.5in}
+\sum_{N\ge c\ge d>0}K(\nu_{a,-b}^{Im},\nu_{c,-d}^{Re})
+\sum_{N\ge c\ge d>0}K(\nu_{a,-b}^{Im},\nu_{c,-d}^{Im})
\\
&=\sum_{N\ge c>d>0}
\Big[K(\nu_{a,-b}^{Im},\mu_{cd}^{Re})+K(\nu_{a,-b}^{Im},\mu_{cd}^{Im})
+K(\nu_{a,-b}^{Im},\nu_{c,-d}^{Re})+K(\nu_{a,-b}^{Im},\nu_{c,-d}^{Im})\Big]
\\
&\hspace{.5in}+\sum_{N\ge c=d>0}
\Big[K(\nu_{a,-b}^{Im},\mu_{cd}^{Re})+K(\nu_{a,-b}^{Im},\mu_{cd}^{Im})
+K(\nu_{a,-b}^{Im},\nu_{c,-d}^{Re})+K(\nu_{a,-b}^{Im},\nu_{c,-d}^{Im})\Big]
\\
&:=\sum_{N\ge c>d>0}D_{upper}+\sum_{N\ge c=d>0}D_{diagonal}
\end{align*}
We have
\begin{eqnarray*}
&&D_{upper}=
\frac{1}{16}\Big[
+32\delta_{a,b}\delta_{a,c}\lambda_a^4
+32\delta_{a,b}\delta_{a,d}\lambda_a^4
-16\delta_{a,c}\lambda_a^4
-24\delta_{a,c}\delta_{a,d}\lambda_a^4
\\&&\hspace{.3in}
+8\delta_{a,c}\delta_{b,d}\lambda_a^4
-16\delta_{a,d}\lambda_a^4
+8\delta_{a,d}\delta_{b,c}\lambda_a^4
-8\delta_{a,c}\delta_{a,d}\lambda_a^2\lambda_b^2
+16\delta_{a,c}\delta_{b,d}\lambda_a^2\lambda_b^2
\\&&\hspace{.3in}
+16\delta_{a,d}\delta_{b,c}\lambda_a^2\lambda_b^2
-8\delta_{b,c}\delta_{b,d}\lambda_a^2\lambda_b^2
+40\delta_{a,b}\delta_{a,d}\lambda_a^2\lambda_c^2
-12\delta_{a,d}\lambda_a^2\lambda_c^2
-8\delta_{b,d}\lambda_a^2\lambda_c^2
\\&&\hspace{.3in}
+40\delta_{a,b}\delta_{a,c}\lambda_a^2\lambda_d^2
-12\delta_{a,c}\lambda_a^2\lambda_d^2
-8\delta_{b,c}\lambda_a^2\lambda_d^2
+8\delta_{a,c}\delta_{b,d}\lambda_b^4
+8\delta_{a,d}\delta_{b,c}\lambda_b^4
\\&&\hspace{.3in}
-16\delta_{b,c}\lambda_b^4
-24\delta_{b,c}\delta_{b,d}\lambda_b^4
-16\delta_{b,d}\lambda_b^4
-8\delta_{a,d}\lambda_b^2\lambda_c^2
-12\delta_{b,d}\lambda_b^2\lambda_c^2
-8\delta_{a,c}\lambda_b^2\lambda_d^2
\\&&\hspace{.3in}
-12\delta_{b,c}\lambda_b^2\lambda_d^2
-16\delta_{a,b}\delta_{a,d}\lambda_c^4
+8\delta_{a,d}\lambda_c^4
+8\delta_{b,d}\lambda_c^4
-16\delta_{a,b}\delta_{a,c}\lambda_d^4
+8\delta_{a,c}\lambda_d^4
+8\delta_{b,c}\lambda_d^4
\Big]
\end{eqnarray*}
and
\begin{eqnarray*}
&&D_{diagonal}=
\frac{1}{16}\Big[
+24\delta_{a,b}\delta_{a,c}\lambda_a^4
-32\delta_{a,b}\delta_{a,c}\delta_{a,d}\lambda_a^4
+24\delta_{a,b}\delta_{a,d}\lambda_a^4
-12\delta_{a,c}\lambda_a^4
\\&&\hspace{.3in}
-16\delta_{a,c}\delta_{a,d}\lambda_a^4
+6\delta_{a,c}\delta_{b,d}\lambda_a^4
-12\delta_{a,d}\lambda_a^4
+6\delta_{a,d}\delta_{b,c}\lambda_a^4
+20\delta_{a,c}\delta_{b,d}\lambda_a^2\lambda_b^2
\\&&\hspace{.3in}
+20\delta_{a,d}\delta_{b,c}\lambda_a^2\lambda_b^2
+60\delta_{a,b}\delta_{a,d}\lambda_a^2\lambda_c^2
-18\delta_{a,d}\lambda_a^2\lambda_c^2
-12\delta_{b,d}\lambda_a^2\lambda_c^2
+60\delta_{a,b}\delta_{a,c}\lambda_a^2\lambda_d^2
\\&&\hspace{.3in}
-18\delta_{a,c}\lambda_a^2\lambda_d^2
-12\delta_{b,c}\lambda_a^2\lambda_d^2
+6\delta_{a,c}\delta_{b,d}\lambda_b^4
+6\delta_{a,d}\delta_{b,c}\lambda_b^4
-12\delta_{b,c}\lambda_b^4
\\&&\hspace{.3in}
-16\delta_{b,c}\delta_{b,d}\lambda_b^4
-12\delta_{b,d}\lambda_b^4
-12\delta_{a,d}\lambda_b^2\lambda_c^2
-18\delta_{b,d}\lambda_b^2\lambda_c^2
-12\delta_{a,c}\lambda_b^2\lambda_d^2
\\&&\hspace{.3in}
-18\delta_{b,c}\lambda_b^2\lambda_d^2
-12\delta_{a,b}\delta_{a,d}\lambda_c^4
+6\delta_{a,d}\lambda_c^4
+6\delta_{b,d}\lambda_c^4
-12\delta_{a,b}\delta_{a,c}\lambda_d^4
+6\delta_{a,c}\lambda_d^4
+6\delta_{b,c}\lambda_d^4
\Big]
\end{eqnarray*}
For $a>b>0$,
\begin{eqnarray*}
&&\sum_{N\ge c>d>0}D_{upper}=
\frac{1}{16}\Big[
-16(a-1)\lambda_a^4
+8\lambda_a^4
-16(N-a)\lambda_a^4
+16\lambda_a^2\lambda_b^2
\\&&\hspace{.3in}
\hspace{.3in}
-12\lambda_a^2\sum_{c=a+1}^{N}\lambda_c^2
-8\lambda_a^2\sum_{c=b+1}^{N}\lambda_c^2
-12\lambda_a^2\sum_{d=1}^{a-1}\lambda_d^2
-8\lambda_a^2\sum_{d=1}^{b-1}\lambda_d^2
+8\lambda_b^4
\\&&\hspace{.3in}
\hspace{.3in}
-16(b-1)\lambda_b^4
-16(N-b)\lambda_b^4
-8\lambda_b^2\sum_{c=a+1}^{N}\lambda_c^2
-12\lambda_b^2\sum_{c=b+1}^{N}\lambda_c^2
\\&&\hspace{.3in}
\hspace{.3in}
-8\lambda_b^2\sum_{d=1}^{a-1}\lambda_d^2
-12\lambda_b^2\sum_{d=1}^{b-1}\lambda_d^2
+8\sum_{c=a+1}^{N}\lambda_c^4
+8\sum_{c=b+1}^{N}\lambda_c^4
+8\sum_{d=1}^{a-1}\lambda_d^4
+8\sum_{d=1}^{b-1}\lambda_d^4
\Big]
\end{eqnarray*}
For $a=b>0$,
\begin{eqnarray*}
&&\sum_{N\ge c>d>0}D_{upper}=
\frac{1}{16}\Big[
+32(a-1)\lambda_a^4
+32(N-a)\lambda_a^4
-16(a-1)\lambda_a^4
-16(N-a)\lambda_a^4
\\&&\hspace{.3in}
\hspace{.3in}
+40\lambda_a^2\sum_{c=a+1}^{N}\lambda_c^2
-12\lambda_a^2\sum_{c=a+1}^{N}\lambda_c^2
-8\lambda_a^2\sum_{c=a+1}^{N}\lambda_c^2
+40\lambda_a^2\sum_{d=1}^{a-1}\lambda_d^2
-12\lambda_a^2\sum_{d=1}^{a-1}\lambda_d^2
\\&&\hspace{.3in}
\hspace{.3in}
-8\lambda_a^2\sum_{d=1}^{a-1}\lambda_d^2
-16(a-1)\lambda_a^4
-16(N-a)\lambda_a^4
-8\lambda_a^2\sum_{c=a+1}^{N}\lambda_c^2
\\&&\hspace{.3in}
\hspace{.3in}
-12\lambda_a^2\sum_{c=a+1}^{N}\lambda_c^2
-8\lambda_a^2\sum_{d=1}^{a-1}\lambda_d^2
-12\lambda_a^2\sum_{d=1}^{a-1}\lambda_d^2
-16\sum_{c=a+1}^{N}\lambda_c^4
\\&&\hspace{.3in}
\hspace{.3in}
+8\sum_{c=a+1}^{N}\lambda_c^4
+8\sum_{c=a+1}^{N}\lambda_c^4
-16\sum_{d=1}^{a-1}\lambda_d^4
+8\sum_{d=1}^{a-1}\lambda_d^4
+8\sum_{d=1}^{a-1}\lambda_d^4
\Big]
\end{eqnarray*}
For $a>b>0$,
\begin{eqnarray*}
&&\sum_{N\ge c=d>0}D_{diagonal}=
\frac{1}{16}\Big[
-12\lambda_a^4
-16\lambda_a^4
-12\lambda_a^4
-18\lambda_a^4
-12\lambda_a^2\lambda_b^2
-18\lambda_a^4
-12\lambda_a^2\lambda_b^2
\\&&\hspace{.3in}
-12\lambda_b^4
-16\lambda_b^4
-12\lambda_b^4
-12\lambda_a^2\lambda_b^2
-18\lambda_b^4
-12\lambda_a^2\lambda_b^2
-18\lambda_b^4
+6\lambda_a^4
+6\lambda_b^4
+6\lambda_a^4
+6\lambda_b^4
\Big]
\end{eqnarray*}
For $a=b>0$,
\begin{eqnarray*}
\sum_{N\ge c=d>0}D_{diagonal}=0
\end{eqnarray*}
Therefore, for $a>-b>0$,
\begin{eqnarray*}
&&\Ric^N(\nu_{ab}^{Im})
=\sum_{N\ge c>d>0}D_{upper}+\sum_{N\ge c=d>0}D_{diagonal}
\\&&\hspace{.3in}
=\frac{1}{16}\Big[
-40\lambda_a^4
-40\lambda_b^4
-32\lambda_a^2\lambda_b^2
-12\lambda_a^2\sum_{d=1}^{a-1}\lambda_d^2
-8\lambda_a^2\sum_{d=1}^{b-1}\lambda_d^2
-8\lambda_b^2\sum_{d=1}^{a-1}\lambda_d^2
\\&&\hspace{.3in}
\hspace{.3in}
-12\lambda_b^2\sum_{d=1}^{b-1}\lambda_d^2
+8\sum_{d=1}^{a-1}\lambda_d^4
+8\sum_{d=1}^{b-1}\lambda_d^4
-16N\lambda_a^4
-16N\lambda_b^4
-12\lambda_a^2\sum_{c=a+1}^{N}\lambda_c^2
\\&&\hspace{.3in}
\hspace{.3in}
-8\lambda_a^2\sum_{c=b+1}^{N}\lambda_c^2
-8\lambda_b^2\sum_{c=a+1}^{N}\lambda_c^2
-12\lambda_b^2\sum_{c=b+1}^{N}\lambda_c^2
+8\sum_{c=a+1}^{N}\lambda_c^4
+8\sum_{c=b+1}^{N}\lambda_c^4
\Big]
\end{eqnarray*}
and, for $a=-b>0$,
\begin{eqnarray*}
&&\Ric^N(\nu_{ab}^{Im})
=\sum_{N\ge c>d>0}D_{upper}+\sum_{N\ge c=d>0}D_{diagonal}
=0
\end{eqnarray*}

\end{proof}

\begin{remark}
Taking $N\to\infty$, we see that the Ricci curvature is negative infinity
in most directions.
\end{remark}

\bibliographystyle{amsplain}

\end{document}